\documentclass[11pt]{article}

\newcommand{\Debug}{0}
\ifnum \Debug = 1 \usepackage[notref,notcite]{showkeys}
\fi

\usepackage{epsfig}
\usepackage{graphicx}
\usepackage{amsbsy}
\usepackage{amsmath}
\usepackage{amsfonts}
\usepackage{amssymb}
\usepackage{textcomp}
\usepackage{hyperref}
\usepackage{aliascnt}

\newcommand{\mcm}[3]{\newcommand{#1}[#2]{{\ensuremath{#3}}}} 

\mcm{\tuple}{1}{\langle #1 \rangle}
\mcm{\name}{1}{\ulcorner #1 \urcorner}
\mcm{\Nbb}{0}{\mathbb{N}}
\mcm{\Zbb}{0}{\mathbb{Z}}
\mcm{\Rbb}{0}{\mathbb{R}}
\mcm{\Cbb}{0}{\mathbb{C}}
\mcm{\Qbb}{0}{\mathbb{Q}}
\mcm{\Bcal}{0}{\cal B}
\mcm{\Ccal}{0}{\cal C}
\mcm{\Dcal}{0}{\cal D}
\mcm{\Ecal}{0}{\cal E}
\mcm{\Fcal}{0}{\cal F}
\mcm{\Gcal}{0}{\cal G}
\mcm{\Hcal}{0}{\cal H}
\mcm{\Ical}{0}{\cal I}
\mcm{\Jcal}{0}{\cal J}
\mcm{\Kcal}{0}{\cal K}
\mcm{\Lcal}{0}{\cal L}
\mcm{\Mcal}{0}{\cal M}
\mcm{\Ncal}{0}{\cal N}
\mcm{\Ocal}{0}{{\cal O}}
\mcm{\Pcal}{0}{{\cal P}}
\mcm{\Qcal}{0}{{\cal Q}}
\mcm{\Rcal}{0}{{\cal R}}
\mcm{\Scal}{0}{{\cal S}}
\mcm{\Tcal}{0}{{\cal T}}
\mcm{\Ucal}{0}{{\cal U}}
\mcm{\Vcal}{0}{{\cal V}}
\mcm{\Wcal}{0}{{\cal W}}
\mcm{\Xcal}{0}{{\cal X}}
\mcm{\Ycal}{0}{{\cal Y}}
\mcm{\Mfrak}{0}{\mathfrak M}

\mcm{\restric}{0}{\upharpoonright}
\mcm{\upset}{0}{\uparrow}
\mcm{\onto}{0}{\twoheadrightarrow}
\mcm{\smallNbb}{0}{{\small \mathbb{N}}}
\DeclareMathOperator{\preop}{op}
\mcm{\op}{0}{^{\preop}}

\newcommand{\se}{\subseteq}
{\begin{array}{c}
\setlength{\unitlength}{1em}}%
{\end{array}}

\usepackage{amsthm}

\newcommand{\theoremize}[2]{\newaliascnt{#1}{thm} \newtheorem{#1}[#1]{#2} \aliascntresetthe{#1}}

\theoremstyle{plain}
\newtheorem{thm}{Theorem}[section]
\theoremize{lem}{Lemma}
\theoremize{skolem}{Skolem}
\theoremize{fact}{Fact}
\theoremize{sublem}{Sublemma}
\theoremize{claim}{Claim}
\theoremize{obs}{Observation}
\theoremize{prop}{Proposition}
\theoremize{problem}{{Problem}}
\theoremize{cor}{Corollary}
\theoremize{que}{Question}
\theoremize{oque}{Open Question}
\theoremize{con}{Conjecture}

\theoremstyle{definition}
\theoremize{dfn}{Definition}
\theoremize{rem}{Remark}
\theoremize{eg}{Example}
\theoremize{exercise}{Exercise}
\theoremstyle{plain}

\usepackage{verbatim}
\usepackage{enumerate}
\usepackage[all]{xy}
\usepackage[percent]{overpic}

\usepackage{subfig}

\title{\scshape Every planar graph with the Liouville property is amenable
}

\newcommand{\sm}{\setminus}
\newcommand{\comm}[1]{}

\renewcommand{\mathring}[1]{#1^\circ}
\newcommand{\Gd}{\ensuremath{G^\diamond}}
\newcommand{\Hd}{\ensuremath{H^\diamond}}
\newcommand{\Gdd}{\ensuremath{G^{*\diamond}}}
\newcommand{\plg}{plane line graph}
\newcommand{\rbt}{roundabout}
\newcommand{\Gr}{\ensuremath{G^\circ}}
\newcommand{\Grd}{\ensuremath{G^{* \circ}}}
\newcommand{\Hr}{\ensuremath{H^\circ}}
\newcommand{\Hrd}{\ensuremath{H^{* \circ}}}
\newcommand{\labtequ}[2]{
 \begin{equation} \label{#1} 	\begin{minipage}[c]{0.9\textwidth}  #2 \end{minipage} 
\ignorespacesafterend \end{equation} } 
\newcommand{\cc}{\ensuremath{\mathcal C}}
\newcommand{\ceil}[1]{\ensuremath{\left\lceil #1 \right\rceil}}
\newcommand{\floor}[1]{\ensuremath{\left\lfloor #1 \right\rfloor}}

\newcommand{\arE}{{E}}

\newcommand{\R}{\ensuremath{\mathbb R}}
\newcommand{\Z}{\ensuremath{\mathbb Z}}

\newcommand{\defi}[1]{{\emph{#1}}}

\newcommand{\knl}{Kirchhoff's node law}
\newcommand{\kcl}{Kirchhoff's cycle law}

\newcommand{\Dhf}{Dirichlet harmonic function}
\newcommand{\rw}{random walk}

\renewcommand{\iff}{if and only if}
\newcommand{\fe}{for every}

\newcommand{\st}{such that}

\newcommand{\ti}{there is}

\newcommand{\g}{\ensuremath{G\ }}
\newcommand{\G}{\ensuremath{G}}

\DeclareMathOperator{\supp}{supp}

\author{Johannes Carmesin\thanks{Supported by EPSRC grant EP/T016221/1.}
\medskip 
\\
{School of Mathematics}
\\
  {University of Birmingham}\\
{B15 2TT, UK} 
\and Agelos Georgakopoulos\thanks{Supported by EPSRC grant EP/L002787/1, 
and  by the European Research Council (ERC) under the European Union's Horizon 2020 research and 
innovation programme (grant agreement No 639046).}
\medskip 
\\
  {Mathematics Institute}\\
 {University of Warwick}\\
  {CV4 7AL, UK}
}

\begin{document}
\maketitle
\begin{abstract}
We introduce a strengthening of the notion of transience for planar maps in order to relax the 
standard condition of bounded degree appearing in various results, in particular, 
the existence of \Dhf s proved by Benjamini \& Schramm. As a corollary we obtain that every planar 
non-amenable graph admits non-constant \Dhf s.
\end{abstract}

\section{Introduction}

A well-known result of Benjamini \& Schramm \cite{BeSchrHar,BeSchrST} states that every transient planar graph with bounded 
vertex degrees admits non-constant harmonic functions with finite Dirichlet energy; we will call such a function a (non-constant) \defi{\Dhf} from now on. In 
particular, such a graph does not have the Liouville property. Two independent proofs of this 
theorem were given in \cite{BeSchrHar,BeSchrST}, one using circle packings and one using square 
tilings. 

The bounded degree condition was essential in both these proofs, and is in fact necessary: consider 
for example a 1-way infinite path where the $n$th edge has been duplicated by $2^n$ parallel edges. Still, there are 
natural classes of unbounded degree graphs where such obstructions do not occur, and it is 
interesting to ask whether the above result can be extended to graphs with unbounded degrees in a 
meaningful way. Recently, planar graphs with unbounded degrees have been 
attracting a lot of interest, in particular due to research on coarse geometry \cite{BenCoa}, random 
walks \cite{ABGN,intersection,planarPB,gwynne_tutte_2018} and random planar maps 
\cite{Angel2018,pshit}. 
Motivated by this, our main result extends the aforementioned result  of Benjamini \& Schramm to 
unbounded degree graphs by replacing the transience condition with a stronger one, which we call 
\defi{roundabout-transience} and explain below:

\begin{thm}\label{thm:UK-trans}\empty
Let $G$ be a locally finite roundabout-transient planar map. Then $G$ admits a non-constant \Dhf.
\end{thm}
A \defi{planar map} \G, also called  a \defi{plane graph}, is a graph endowed with an embedding in 
the plane.
The \defi{roundabout graph} $\mathring{G}$ is obtained from \g by replacing each vertex $v$ with a 
cycle $\mathring{v}$ in such a way that the edges incident with $v$ are incident with distinct 
vertices of $\mathring{v}$ (of degree 3), preserving their cyclic ordering; see also 
\autoref{secUKtr}.
We say that \g is \defi{roundabout-transient} if $\mathring{G}$  is transient.\footnote{The authors 
coined this term in Warwick, UK, where there are many roundabouts.} In \autoref{secUKtr} we relate 
\Gr\ 
with circle packings of $G$. We show that roundabout-transience implies transience in Lemma~\ref{roundabout_rem}.

\begin{eg}
The aforementioned example of a 1-way infinite path with the $n$-th edge replaced by $2^n$ edges, is 
transient, but not roundabout-transient. Indeed, each roundabout $\mathring{v}$ contains a cut 
consisting of just two edges 
separating the root from infinity. Thus the effective resistance to infinity is  infinite in the 
roundabout graph, and Lyons' criterion (\autoref{lyons}) implies recurrence. 
\end{eg}

We also provide a further way to strengthen the transience condition so as to guarantee Dirichlet 
harmonic functions.
The idea is to require that there is a flow $f$ from some vertex having not only finite Dirichlet 
energy, 
as required by Lyons' criterion, but also a finite norm in a different Hilbert space, obtained by 
giving weights to the edges depending on the degrees of their end-vertices. This is made precise in 
the following corollary of  \autoref{thm:UK-trans}, which we deduce in \autoref{sec9}.

\begin{cor}\label{supersuper-trans_mod}
Let $G$ be a locally finite planar graph such that there is a flow $f$ from 
some vertex $x$ such that 
$${\sum_{vw\in E(G)} [deg(v)^2+deg(w)^2]f(vw)^2}< \infty.$$
Then $G$ admits a non-constant Dirichlet harmonic function.
\end{cor}

As shown in \autoref{grid_minor}, the order of magnitude of the weights here is best-possible. 
Hence \autoref{supersuper-trans_mod} is tight, which indicates a way in which 
\autoref{thm:UK-trans} is tight too. We prove it in Section~\ref{sec9} (see Corollary~\ref{super-trans}).

Our work was partly motivated by a problem from \cite{planarPB}, asking whether every simple planar 
graph with the Liouville property is (vertex-)\defi{amenable},\footnote{For bounded degree graphs, 
vertex-non-amenability and the related notion 
of edge-non-amenability agree. For graphs with unbounded degrees like ours this is no longer the 
case, 
and we always mean vertex-non-amenable when writing non-amenable.} by which we mean that for every 
$\epsilon>0$ there is a finite set $S$ of vertices 
of \g such that less than $\epsilon |S|$ vertices outside $S$ have a neighbour in $S$. As we show 
in 
\autoref{secApps},
\begin{thm}\label{nonam_UK-trans_intro}
 Every locally finite non-amenable planar map is roundabout-transient.
\end{thm}
Combining this with \autoref{thm:UK-trans} yields a positive answer to the aforementioned problem, 
and much more. This strengthens a result of Northshield \cite{NorGeo}, stating that every bounded 
degree non-amenable planar graph admits non-constant bounded 
harmonic functions, in two ways: it relaxes the bounded degree condition, and provides Dirichlet 
rather than bounded harmonic functions. 

Benjamini \cite{BenCoa} constructed a bounded degree non-amenable graph with the  Liouville 
property. The last result shows that such a graph cannot be planar even if we drop the bounded 
degree assumption.

\medskip
We think of Theorems~\ref{thm:UK-trans} and~\ref{nonam_UK-trans_intro} as indications that the 
notion of roundabout-transience is satisfied in many cases, and has strong implications. We expect 
it to find further applications. For example, we expect that the results of \cite[Section 2]{NorCir} 
generalise from bounded-degree 
non-amenable planar maps to roundabout-transient ones. Moreover, one could try to extend the main 
results of \cite{planarPB} and  \cite{ABGN}, which identify the Poisson boundary of planar graphs 
with the 
boundary of the square tiling, and the circle packing respectively, from the bounded-degree 
transient case to the roundabout-transient case, as we did in this paper for the result of Benjamini 
\& Schramm on Dirichlet harmonic functions. Finally, perhaps the most interesting problem of this 
form is the following:
\begin{problem} \label{probl}
Let \g be the 1-skeleton of a triangulation of an open disc in $\R^2$. Is it true that \g admits a 
circle packing 
in the unit disc if and only if it has a roundabout-transient subgraph\footnote{Here we follow the 
convention that subgraphs of plane graphs are endowed with the induced embedding, and are also 
plane graphs.}?
\end{problem}
If true, this would extend a well-known theorem of He \& Schramm \cite{HeSchrHyp}, stating that if 
\g is recurrent, then it admits a circle packing whose carrier is the whole plane, and if it is 
transient and has bounded degrees, then it admits a circle packing in the unit disc. (It is known 
that every 1-skeleton of a triangulation of an open disc admits a circle packing in either the whole plane or the unit disc, but not in 
both \cite{HeSchrFix,HeSchrHyp,SchraRig}.) The reason why we do not conjecture that \g admits a 
circle packing in the unit disc if and only if it is roundabout-transient itself in 
Problem~\ref{probl}, is that given any circle packing in the unit disc, it is always possible to 
insert enough new discs to make the contacts graph roundabout-recurrent. We leave this as an 
exercise for the interested reader.
\medskip

\begin{figure}

\centering 
\begin{overpic}[width=.5\linewidth]{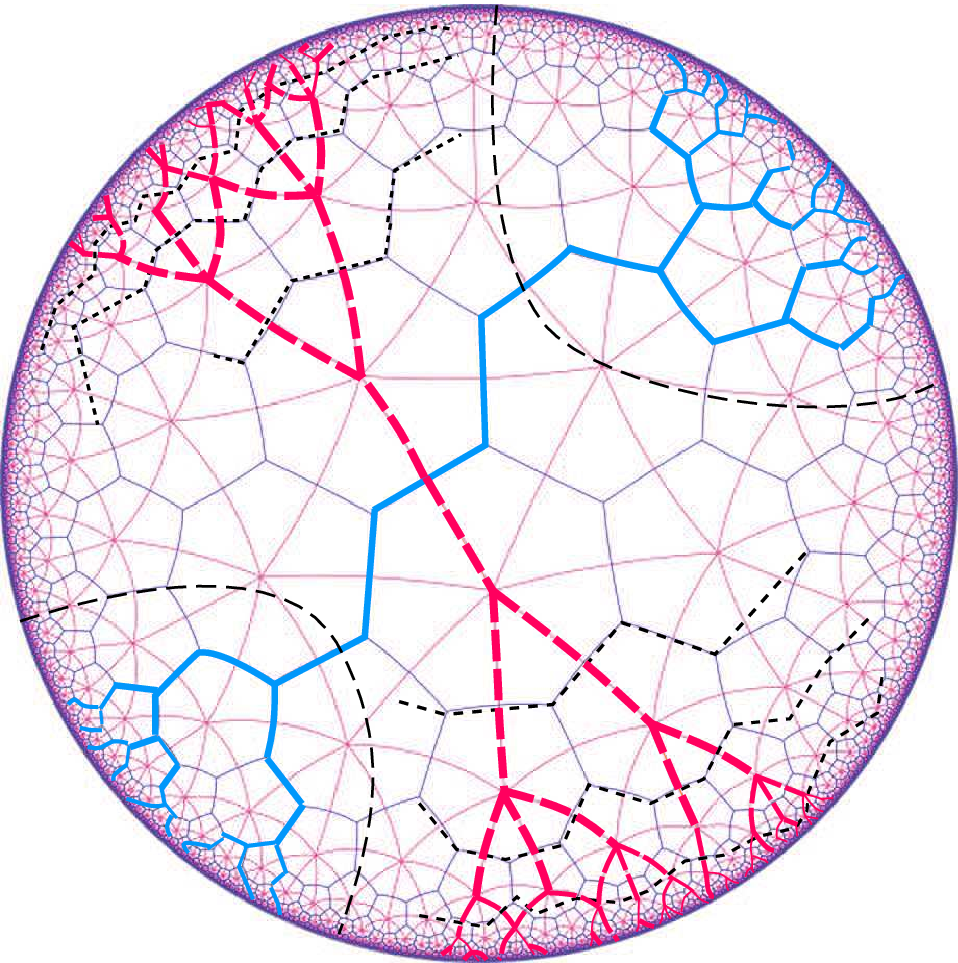} 
\put(22,23){$T_1$}
\put(75,70){$T_2$}
\end{overpic}
 \caption{\small A tesselation of the hyperbolic plane by 7-gons, depicted in black, and its dual 
graph, depicted in light purple. In blue we colour the edges in the support of a flow in the 
tesselation, in red we colour the edges in the support of a flow in the dual. 
The two subgraphs $T_1,T_2$ delimited by the dashed smooth curves are transient because of 
the blue flow. The dual of the red flow (indicated by dashed non-smooth paths) witnesses the fact 
that the effective 
conductance between $T_1$ and $T_2$ is finite because this dual flow has finite energy.} 
\label{figfh}

\end{figure}

We now give an overview of the proof of \autoref{thm:UK-trans}. 
As shown in \cite{C:harmonic}, a graph admits non-constant \Dhf s if and only if it has two 
disjoint 
transient subgraphs $T_1,T_2$ such that the effective conductance between $T_1$ and $T_2$ is 
finite; 
see \autoref{flow_pot_char}. To show that our graphs satisfy this condition, we start with a flow 
provided by Lyons' criterion. This flow lives in an auxiliary graph 
which for the purposes of this illustration can be thought of as a superimposition of $G$ and its 
dual. We split this flow into four sub-flows, supported in disjoint regions of the plane, using the 
square tiling techniques of 
\cite{planarPB}.  We use two  of these subflows to obtain $T_1,T_2$, and we apply a duality argument 
 to the 
other two to show that the effective conductance between $T_1$ and $T_2$ is finite; see 
\autoref{figfh} and  \autoref{4trans_graphs}. 
The latter step can be thought of as a  variation of the idea that the effective resistance from 
the 
top to the bottom of a rectangle equals the effective conductance  from left to right, with the 
latter two subflows showing finiteness of the 
top-to-bottom effective resistance. 
 
\medskip
The idea of handling unbounded-degree graphs by first applying a transformation into a bounded-degree graph ---in our case, the roundabout graph--- also appears in \cite{GurNachRec}, where the  transformation used is the `star-tree transform'.

Sections~\ref{prelims}--\ref{secUKtr} contain definitions and preliminaries about graphs and random walks, harmonic functions, and Roundabout-transience respectively. The two crossing flows of the above sketch are constructed in Section~\ref{sec ST}, after an introduction into square tilings which are used as a tool. In Section~\ref{sec Hfpg} we obtain a general criterion for the existence of non-constant Dirichlet harmonic functions in planar graphs, and use it to prove \autoref{thm:UK-trans} in Section~\ref{sec proof}. We deduce \autoref{nonam_UK-trans_intro} in Sections~\ref{secApps}--\ref{sec9}.

\section{Preliminaries}\label{prelims}

\subsection{Graphs} \label{secGraphs}

We follow the terminology of \cite{diestelBook05noEE} for graph-theoretic terms unless otherwise 
stated.
A \defi{graph} $G$ is a pair $(V,\arE)$ where $V$ is 
the set of vertices of $G$, and $\arE$ is a set of directed pairs of elements of $V$, called the 
(directed)
\defi{edges} of \G. (Although we are studying undirected graphs, we follow the latter convention for convenience in dealing with flows.)

All our graphs are \defi{simple}:  they have no loops or  parallel edges. (In the few occasions 
where we 
contract edges, one can subdivide any resulting parallel edges or loops to stay within the class of 
simple graphs.) 
A graph is \emph{locally finite} if all its vertices have finite \defi{degree}, where the degree of 
a vertex is the number of incident edges. Most graphs in this paper 
are locally finite. 
A locally finite graph $G$ is \emph{1-ended} if for every finite vertex $S$, the graph $G-S$ 
(obtained from $G$ by deleting the vertices in $S$ and their incident edges) has 
only one infinite component.
Given a vertex set $X$, by $E(X)$ we denote those edges with both endvertices in $X$. 

A \emph{cut} of a graph $G$ is the set of edges between a set of vertices $U\subseteq V(G)$ and its 
complement $V(G) \sm U$.

\subsection{Plane graphs} \label{secPlane}

A graph is \defi{planar}, if it admits an embedding in 
the plane $\R^2$. A \defi{plane graph} is a (planar) graph endowed with a fixed embedding in 
the plane. We will be using the notion of the dual of a plane graph in the standard sense, but we 
adapt it to our directed graphs so that the directions of the edges of the primal determine the 
directions of the edges of the dual as follows.
The \defi{dual} of a plane (directed) graph $G=(V,\arE)$ is the graph $G^*=(F,\arE^*)$ 
whose vertex set is the set $F$ of faces of $G$, having an edge $e^*$ from a face $v$ to a face $w$ 
whenever \g has an edge $e$ incident with both $v$ and $w$ \st\ $v$ lies on the right of $e$ as we 
move along the direction of $e$. Note that by drawing the vertices of $G^*$ inside the corresponding 
faces of \g we can obtain an embedding in $\R^2$ \st\ $G^{**}=G$. (To be more precise, $G^{**}$ is $G$ with all edge directions reversed.)
To simplify notation we will, with a slight abuse, suppress the bijection ${\cdot}^*$ between the 
edge sets $\arE,\arE^*$ of two dual 
plane graphs and pretend that $\arE=\arE^*$.

We will be using the following simple fact about plane dual graphs. A \emph{bond} is a minimal 
nonempty cut; that is, a cut that does not include any other nonempty cut. 
\begin{lem}[{\cite[Proposition~4.6.1]{diestelBook05noEE}}] \label{lf bonds}
Let $G$ and $G^*$ be dual plane graphs, and suppose they are both locally finite. Then every finite  bond 
$b$ of \g forms a cycle $C$ in $G^*$ such that one of the components of $G - b$ lies in the 
interior  of $C$ and the 
other in its exterior.\footnote{The statement that the components of $G - b$ lie in distinct sides 
of 
$C$ is given in the proof of \cite[Proposition~4.6.1]{diestelBook05noEE} rather than in its 
statement. Although the latter is assuming the graphs to be finite, it can be easily adapted to our setup by considering an appropriate finite subgraph of $G$ containing $b$.}
\end{lem}

\subsection{Electrical currents} \label{secENB}

Given a graph $G=(V,\arE)$ and a function $i: \arE \to \R$, the {divergence} $i^*(x)$ of $i$ at a 
vertex $x$ is the net flow out of $x$, that is, \\
$i^*(x):=\sum_{xy\in \arE} i(xy) - \sum_{zx\in \arE} i(zx)$.
We say that $i$ satisfies \defi{\knl} at $x$ if  $i^*(x)=0$. 

A  \emph{divergence free flow} is a function $i: \arE \to \R$ satisfying \defi{\knl} at every 
vertex. 
In an infinite graph it is possible for $i$ to satisfy \knl\ at all vertices except 
a single vertex $o$, at which we have $i^*(o)\neq 0$. In this case $i$ is called a \defi{flow from} 
$o$ (to infinity). The \defi{intensity} of $i$ is the divergence $i^*(o)$. 
For a finite vertex-set $A$, we say that $i$ is a  \defi{flow from $A$} if $i^*(x)>0$ \fe\ $x\in A$ 
and $i^*(x)=0$ \fe\ $x\not\in A$. The \defi{support} $supp(i)$ of $i$ is the edge set $\{e\in \arE 
\mid i(e)\neq 0 \}$.

A \defi{potential} on  $G$ is a function $u: V\to \R$. The difference operator $\partial$ turns each 
potential $u: V\to \R$ into a function $\partial u: \arE \to \R$ by letting $\partial u(xy) := 
u(x)-u(y)$. 
If $\partial u$ satisfies Kirchhoff's node law, then $u$ satisfies 
the discrete Laplace equation:\footnote{This can be seen by solving the equation  $\sum_{y\in  
N(x)}(u(x) - u(y)) =0$ for 
$u(x)$. }
\labtequ{harm}{$u(x) = \frac{\sum_{y\in  N(x)} u(y)}{deg(x)}$,}
where $N(x)$  denotes the set of neighbours of $x$, and $deg(x)$ the cardinality of $N(x)$. 
If $u$ satisfies \eqref{harm}, then we say that $u$ is \defi{harmonic} at 
the vertex $x$. Note that the above implication can be reversed to yield that if a potential $u$ is 
harmonic, then $\partial u$ satisfies Kirchhoff's node law.

A potential $u: V\to \R$ is \defi{harmonic} if it is harmonic at every vertex $x\in V$. 
The (Dirichlet) \defi{energy} of a function  $i: \arE \to \R$ is defined by
$${\cal E}(i):= \sum_{e\in \arE} i^2(e).$$

The \emph{energy of a potential} $u$ is the energy of $\partial u$; in formulas:
${\cal E}(u):= \sum_{xy\in \arE}  \left(u(x)-u(y)\right)^2$. 
A  harmonic function with finite Dirichlet energy is called a \defi{Dirichlet harmonic function}. 
A graph has the \emph{Liouville Property} if all of its bounded harmonic functions are constant. It 
is well-known that the Liouville Property implies that all Dirichlet harmonic function are constant 
(indeed, if there is a non-constant Dirichlet harmonic function, then the free-current and the 
wired current do not agree and their difference is a (non-constant) \emph{bounded} Dirichlet 
harmonic function, see \cite{LyonsBook} for details). 

\medskip

A \defi{walk} in a graph \g is a sequence $\{v_0, e_1, \ldots e_k, v_k\}$ alternating between 
vertices and incident edges, 
starting and ending with a vertex. A walk is \defi{closed} if its starting vertex $v_0$ is equal to 
its 
ending vertex $v_k$.  Given a function $i: \arE \to \R$ and a closed walk $W=\{v_0, e_1, \ldots e_k, 
v_k\}$, we define $curl_i(W):= \sum_{j\leq k} (-1)^{\delta_j} i(e_j)$, where $\delta_j=1$ if $W$ 
traverses $e_j$ against its direction, and $\delta_j=0$ otherwise.
We say that $i$ satisfies \emph{Kirchhoff's} cycle law if $curl_i(W)=0$ for every closed walk $W$ 
in $G$ (equivalently, if $curl_i(W)=0$ for every closed walk $W$ visiting no vertex 
-- other than its starting vertex -- more than once).
It is not hard to check that 

\begin{obs} \label{kcl pot}{A flow $i$ satisfies Kirchhoff's cycle law if and only if there is a 
potential $u$ with $i=\partial u$.}
\end{obs}

\subsection{Random walks} \label{secRWB}
All random walks in this paper are simple and take place in discrete time, that is, if the 
random walker is at a vertex $v$ of our graph $G$ at time $n$, then at time $n+1$ it is at a 
neighbour of $v$ chosen uniformly at random. The 
starting vertex of our random walk will always be deterministic, and usually denoted by $o$.

A connected graph \g is \defi{transient} if random walk on \g almost surely visits any fixed 
vertex finitely often. If \g is not transient then it is \defi{recurrent}. The following classical 
result of T.~Lyons characterises 
transience in terms of flows.
\begin{thm}[\cite{lyons}, see also \cite{LyonsBook}]\label{lyons}
A connected locally finite graph \g is transient \iff\ for some (and hence for every) vertex $o\in 
V(G)$,  there is a flow from $o$ in  \g with finite energy.
\end{thm}

Given a transient graph \g and a vertex $o$, we can define a flow $i=i(o)$ from $o$ as 
follows. For every vertex $v\in V$, let $h(v)$ be the probability $p_v(o)$ that random walk from 
$v$ 
will ever reach $o$. In particular, $h(o)=1$. By construction the potential $h$ is harmonic at 
every $v\neq o$. Let $i= \partial h$. By our discussion in \autoref{secENB}, $i$ is a flow 
from $o$, and we call it the \defi{\rw\ flow} from $o$.

\section{Dirichlet harmonic functions}

In this section we explain some of the tools we use in our proofs. 
The following results characterise the locally finite graphs admitting
non-constant Dirichlet harmonic functions. We write ${\Ocal}_{HD}$ for the class of graphs on which 
all \Dhf s are constant.

\begin{thm}[\cite{C:harmonic}]\label{flow_pot_char}
A locally finite graph $G$ admits a
non-constant Dirichlet harmonic function if and only if there are two transient, 
vertex-disjoint, subgraphs $A, B$ of \g such that 
there is a potential  of finite energy which is constant on each of $A$ and $B$ but not on $A\cup 
B$.
\end{thm}

\begin{obs}\label{obs1717}
 By adding a finite path to the subgraph $A$ in \autoref{flow_pot_char} if necessary, and adapting 
the values of the potential on that path, we may assume that 
in the statement of \autoref{flow_pot_char} we moreover have an edge joining a vertex
of $A$ to a vertex of $B$. 
\end{obs}

The following is a variant of \autoref{flow_pot_char} that is more convenient for our purposes in 
this paper.

\begin{cor}\label{flow_pot_char_reform} 
A locally finite graph $G$ admits a
non-constant Dirichlet harmonic function if and only if it admits a divergence free flow $f$ and
a potential $\rho$,  both of finite energy, such that the supports of $f$ and 
$\partial \rho$ 
intersect in precisely one edge.
\end{cor}

\begin{proof}
To prove the forward implication, assuming that $G$ is not in ${\Ocal}_{HD}$, 
\autoref{flow_pot_char} and 
\autoref{obs1717}, yield transient 
vertex-disjoint subgraphs $A,B$, connected by an edge $e$, as well as a potential $\rho$ of finite 
energy which is constant on each of $A,B$ but takes different 
values on them. Using the transience of $A$ and $B$ and \autoref{lyons} it is straightforward 
to construct a divergence free flow $f$ of finite energy that is supported on the edges of $A \cup 
B$ and the 
edge $e$, with $f(e)\neq 0$. The supports of $f$ and  $\partial \rho$ then intersect only in the 
edge $e$ as desired.

\medskip
The backward implication can be shown using the methods of the proof 
of \autoref{flow_pot_char} in \cite{C:harmonic}.\footnote{More precisely, from the existence of $f$ 
and $\rho$ as in that theorem, one can construct transient subgraphs $A$ and $B$ as in 
\autoref{flow_pot_char}.} Here we take a different route; we will give a new 
functional analytic proof. 


We consider the (real) Hilbert space $H$ of functions from $\arE(G)$ to $\R$ with finite Dirichlet 
energy; our scalar product is defined by\\ 
$\langle f\mid g \rangle:=\sum_{e\in \arE(G)}f(e)g(e)$.

The \defi{cycle space} $C$ of \g is the closed span of the subspace of $H$ generated by the cycles 
of \G; that is, for 
each cycle $C_i$ of \G, we let $f_i$ be a non-zero divergence free flow supported on the edges of 
$C_i$ ($f_i$ is determined by $C_i$ up to a multiplicative constant that does not matter), and let 
$C$ be the subspace of $H$ generated by all the $f_i$ under infinite convergent sums.

The \defi{star space} $D$ of \g is the closed span of the subspace of $H$ generated by the atomic 
cuts of \G:
for 
each vertex $v_i$ of \G, we let $a_i$ be the indicator on its incident edges, and let 
$D$ be the subspace of $H$ generated by all the $a_i$ under infinite convergent sums.

Note that $C$ and $D$ are orthogonal spaces, since each cycle satisfies Kirchhoff's first law. 
Moreover, every divergence free flow lies in $D^\perp$: it is straightforward to check that 
$f\in D^\perp$ \iff\ $f$ satisfies \knl\ at every vertex.
Furthermore,  $C^\perp$ coincides with the space $\{\partial u \mid u \text{ is a potential}\}$. 
Therefore, to show that \g admits a non-constant Dirichlet harmonic 
function, it suffices to show that $D^\perp\cap C^\perp$ is non-trivial, as all functions in 
$D^\perp$ satisfy \knl, and so their corresponding potentials are  harmonic by the discussion in 
\autoref{secENB}.

Let us apply these observations to the functions $f$ and $\rho$ of the statement. The assumption 
that  $f$ and 
$\partial \rho$ intersect in precisely one edge implies that $\langle f\mid \partial \rho \rangle 
\neq 0$. 

As $D$ is orthogonal to $C$, we have $C\subseteq D^\perp$, and so 
we can decompose $D^\perp$ as  $D^\perp  =C+(D^\perp\cap C^\perp)$. 
Thus we can write our $f\in D^\perp$ as $f_1+f_2$ with $f_1 \in D^\perp\cap C$ and $f_2\in 
D^\perp\cap C^\perp$. Since $\partial \rho \in C^\perp$, we have $\langle f_1\mid \partial \rho 
\rangle = 0$, and since $\langle f\mid \partial \rho \rangle \neq 0$ we must have $\langle f_2\mid 
\partial \rho \rangle \neq 0$. In particular, $f_2\neq 0$ and so we have proved our claim that $ 
D^\perp\cap C^\perp\ni f_2$ is non-trivial.
\end{proof}

This way we obtain an alternative proof of the following result of Soardi. 

\begin{cor}[\cite{SoardiBook}]\label{fin_sep}
Let $G$ be a locally finite graph with a finite cut $b$ such that $G-b$ has two transient 
components. Then $G$ is not in ${\Ocal}_{HD}$.
\end{cor}

\begin{proof}
We apply \autoref{flow_pot_char}, with $\rho$ being e.g.\ the potential defined by $\rho(x)=i$ for 
every $x$ in $C_i$, where $C_i$ is the $i$th component of $G-b$ in some enumeration of those 
components.  
\end{proof}

\begin{dfn} \label{def wit}
Given a locally finite graph $G$, and a subgraph $H\subseteq G$, we will say that a function $f: 
{E}(G) \to \R$ 
\emph{witnesses that $H$ is transient}, if the restriction 
$f_H$ of $f$ to ${E}(H)$ is a flow from some finite vertex set (to infinity) with finite 
energy.
\end{dfn}
As we can easily modify $f_H$ at finitely many edges to turn it into a flow from a single vertex (to 
infinity),
such an $f_H$ implies that $H$ is transient by \autoref{lyons}.

\begin{obs}\label{potential_or_harmonic}
Let $G$ and $G^*$ be locally finite dual plane graphs. 
Let $f$ be a divergence free flow in $G$ with finite energy. Then at least one of the following is 
true.
\begin{enumerate}[A)]
 \item \label{ph a} The function $f$ satisfies \kcl\ in $G^*$;
\item \label{ph b} there is a finite cut $c$ of $G$ such that $f$ witnesses that at least two 
components of 
$G-c$ are 
transient.
\end{enumerate}
 \end{obs}

\begin{proof}
Suppressing the bijection ${\cdot}^*$ between the directed edges of $G$ and $G^*$, the function $f$ 
can be thought of as a function on $\arE(G^*)$. If $f$ fails to satisfy (A), then there is a finite 
cycle $C$ of $G^*$ at which $f$ violates \kcl. Since $G$ and $G^*$ are dual, the 
edges of $C$ form a cut $C^*$ of $G$, separating it into two subgraphs $U,W$. Moreover, our 
assumption on $C$ means that the net flow of $f$ from $U$ to $W$ is non-zero. Thus $f_U$ is a flow 
from a finite set (namely, from those vertices of $U$ incident with an edge in $C^*$) witnessing 
that $U$ is transient. Similarly, $f_W$ witnesses that $W$ is transient too. 
\end{proof}

\begin{rem}
For finite plane dual graphs $G$ and $G^*$, a function $f$ satisfies Kirchhoff's cycle law in the 
graph $G$ if and only if it satisfies Kirchhoff's node law in the dual graph $G^*$. 
\autoref{potential_or_harmonic} could be understood as an extension of this fact. 
\end{rem}

\section{Roundabout-transience} \label{secUKtr}

The \emph{roundabout graph $\mathring{G}$} of a locally finite plane graph $G$ is 
obtained from $G$ by replacing each vertex $v$ by a cycle (\defi{roundabout}) of length equal to the 
degree of $v$ 
so that every vertex gets degree $3$; formally, the vertex set of $\mathring{G}$ is the set of 
pairs 
$(v,e)$ where $e$ is an edge and  $v$ is an endvertex of $e$. The embedding of $G$ defines the 
(clockwise, say) cyclic 
order $C_v$ on the set of edges incident with the vertex $v$. The edges of $\mathring{G}$  are of 
two types; for each edge $e=\overrightarrow{vw} \in \arE(G)$, we have an edge in $\mathring{G}$ from 
$(v,e)$ to $(w,e)$. For any two consecutive edges $e$ 
and $f$ in the cyclic order $C_v$, we have an edge in $\mathring{G}$ from $(v,e)$ to $(v,f)$.

The roundabout graph $\mathring{G}$ has a canonical embedding in the plane, namely, the one that 
induces the embedding of $G$ when we contract each roundabout into a single vertex. 

With a slight abuse of notation, we will treat $\arE(G)$ as a subset of $\arE(\mathring{G})$, with 
the understanding that 
$e=\overrightarrow{vw} \in \arE(G)$ is identified with $\overrightarrow{(v,e) (w,e)} \in 
\arE(\mathring{G})$.

We say that a graph $G$ is  \emph{roundabout-transient} if  $\mathring{G}$ is 
transient. 

\begin{obs}\label{roundabout_cut}
Every cut of $G$ is a cut of $\mathring{G}$.

Conversely, every cut $b$ of $\mathring{G}$ with $b\se E(G)$ is also a cut of 
  $G$.
\qed
\end{obs}

\begin{rem}\label{plane_not_planar}
The structure of $\mathring{G}$ depends on the chosen embedding of $G$. Here, we construct a planar 
graph $G$ that has both a transient and a recurrent roundabout graph (corresponding to different 
embeddings).

Let $G$ be the graph obtained from the infinite binary tree\footnote{The \emph{binary tree} is the 
unique infinite tree in which every vertex except for the root has degree three, and the root has 
degree two.} $T_2$ by attaching $2^n$ 
leaves at each vertex at distance $n$ from a fixed root of $T_2$. 
Let $G_1$ be the plane graph obtained by embedding  $G$ in the plane in such a way that
all leaves attached to $v$ are embedded consecutively for every $v\in V(T_2)$. It is straightforward 
to check that  $\mathring{G_1}$ is transient: by deleting all leaves of $\mathring{G_1}$ and their incident vertices (in the roundabouts) we 
obtain a subgraph of $\mathring{G_1}$ quasi-isometric to $T_2$. Thus $\mathring{G_1}$ is transient since $T_2$ is.  Let $G_2$ be  the plane graph 
obtained by embedding  $G$ in the plane in such a way that  the  leaves we attached at each vertex are 
separated into three equal subsets by the edges of $T_2$. It is 
not hard to check that $\mathring{G_2}$ is recurrent: the leaves now have the effect of introducing 
exponentially long subdivisions to the edges of $T_2$ at a certain distance from the root. 

To summarise, roundabout-transience is a property of plane graphs and not of planar 
graphs.
\end{rem}

\begin{lem}\label{roundabout_rem}
If $\mathring{G}$ is transient, then so is $G$. 
\end{lem}

\begin{proof}
Since $\mathring{G}$ is transient, it admits a flow  $f$ of finite energy from 
some vertex $o\in V(\mathring{G})$ by Lyons' criterion \autoref{lyons}. We will show that $f$ 
induces a flow of finite energy in $G$.

For a vertex $v\in V(\mathring{G})$, let us denote by $\mathring{v}$ the set of vertices lying in 
the same roundabout as $v$. Note that $f$ satisfies \knl\ at every vertex-set $\mathring{v}$ except 
$\mathring{o}$. Therefore, the restriction $f'$ of $f$ to the edges of $G$ satisfies 
\knl\ at every vertex of $G$ except the vertex $o'$ that gave rise to $\mathring{o}$. In other 
words, $f'$ is a flow from 
$o'$. Its energy is bounded from above by that of $f$, and so $G$ is transient by 
\autoref{lyons}.
\end{proof}

In the following we will often use the notation $\Grd$, by which we mean the 
roundabout graph $\mathring{(G^*)}$ of the dual $G^*$ of the plane graph $G$. For the rest of this section we assume $G^*$ to be locally finite.

The {\em \plg} \Gd\ of a plane graph $G$ is the plane graph obtained from \Gr\ 
by contracting all (non-roundabout) edges of $G$. Another way to define \Gd, explaining its name, 
is 
by letting the vertex set of \Gd\ be the set of midpoints of edges of $G$ and joining two such 
points with an arc whenever the corresponding edges are incident with a common vertex $v$ of $G$ 
and 
lie in the boundary of a common face of $v$. It is clear from this definition that 
\labtequ{Gdd}{$\Gd=(G^*)^\diamond =:\Gdd$.}
A third equivalent definition of \Gd\ can be given by considering a circle packing $P$ of $G$, 
letting $V(\Gd)$ be the set of intersection points of circles of $P$, and letting the arcs in $P$ 
between these points be the edges of \Gd.
A fourth definition of \Gd\ is as the dual of the bipartite graph $G'$, with $V(G')$ consisting of 
the vertices and faces of $G$, and $E(G')$ joining each vertex of $G$ to each of its incident faces.

It is easy to see that $\Gd$ is quasi-isometric (in fact Bilipschitz-equivalent) to \Gr. Since both 
graphs have bounded degrees, \autoref{lyons} easily implies the following (see e.g.\ \cite[Theorem~2.17]{LyonsBook})
\begin{lem}\label{roundabout_plg}
Let $G$ be a locally finite plane graph. Then $\mathring{G}$ is transient if 
and only if $\Gd$ is. \qed
\end{lem}

\autoref{roundabout_plg}, combined with the fact that $\Gd=\Gdd$ \eqref{Gdd}, yields
that if $\mathring{G}$ is transient, then so is $\Grd$. Another way to state this is:
\labtequ{roundabout_dual}{$G$ is roundabout-transient if and 
only if $G^*$ is.}
Combining this with \autoref{roundabout_rem}, we obtain
\begin{cor}\label{dual trans}
If $\mathring{G}$ is transient, then so is $G^*$.
\end{cor}

\begin{proof}
By \autoref{roundabout_rem}, it suffices to show that the roundabout graph of the dual graph $G^*$ 
is transient. 
Recall that the planar line graph $\Gd$ of $G$ is equal to the planar line graph of the dual graph 
$G$. Thus by  \autoref{roundabout_plg} applied twice to the graph $G$ and the graph $G^*$, we 
deduce that the roundabout graph of $G^*$ is transient. 
\end{proof}

\begin{rem}
 It is a simple exercise to check that every bounded-degree transient plane graph is 
roundabout-transient.
\end{rem}

\section{Square tilings and the two crossing flows} \label{sec ST}

\subsection{Square tiling basics}
In this section we use the theory of square tilings of (locally finite) transient plane graphs in 
order to find the 
special flows in our roundabout-transient $G$ mentioned in the introduction. These square tilings 
were introduced in \cite{BeSchrHar}, and generalise a 
classical construction of Brooks et. al.\ \cite{BSST} from finite plane graphs to infinite 
transient ones. 

Let $\cc$ denote the cylinder $(\R / \Z) \times \{0,1]$, or more generally, a cylinder $(\R / \Z) 
\times \{0,a]$ for some real $a>0$ (which will turn out to coincide with the effective resistance 
from a 
vertex $o$ to infinity).\footnote{Throughout this paper we use $\{a,b]$ to denote the half-open 
interval between $a$ and $b$ (which contains $b$ but not $a$).}
A \defi{square tiling} of a plane graph $G$ is a mapping $\tau$ assigning to each edge $e$ of $G$ a 
square $\tau(e)$ contained in \cc, where we allow $\tau(e)$ to be a `trivial square' consisting of 
just a point (see \autoref{figST} for an example). A nice property of square tilings is that every 
vertex $x\in V$ can be associated with a horizontal line segment $\tau(x)\subset \cc$ such that for 
every edge $e$ incident with $x$, $\tau(e)$ is tangent to $\tau(x)$, i.e.\ one of the sides of $\tau(e)$ is contained in $\tau(x)$. 

The construction of this $\tau$ is based on  the \rw\ flow $i$ from a root vertex $o$  (as defined 
in \autoref{secRWB}): the side length of the square $\tau(e)$ is chosen to be $|i(e)|$, and the 
placement of that square inside $\cc$ is decided by a coordinate system where potentials of 
vertices 
induced by the flow $i$ are used as coordinates. For example, the top circle of the cylinder $\cc$ 
is the `line segment' $\tau(o)$ corresponding to $o$, because $o$ has the highest potential. All 
other 
vertices and edges accumulate towards the base of $\cc$, because their potentials (which equal the 
probability for random walk from those vertices to return to $o$, normalised by the height of \cc) 
converge to 0; see 
\cite{planarPB} for details.
\medskip 

We let $w(\tau(e))$ denote the width of the square $\tau(e)$.
Our square tilings always have the following properties which we will use below:
\begin{enumerate}[I)]
\item \label{sti} Two of the sides of $\tau(e)$ are always parallel to the boundary circles of \cc;
\item \label{stii} $w(\tau(e))= |i(e)|$ for every $e\in \arE$, where $i$ denotes the \rw\ flow out 
of $o$;
\item \label{stiii} the interiors of any two such squares $\tau(e) , \tau(f)$ are disjoint;
\item \label{stiv} every point of \cc\ lies in $\tau(e)$ for some $e\in E$;
\item \label{stv} every vertex $x$ can be associated with a horizontal line segment
\footnote{$\tau(x)$ might be a full horizontal circle of \cc. This is always the case for $x=o$.} 
$\tau(x)\subset 
\cc$ so that for every edge $e$ incident with $x$, the square $\tau(e)$ is tangent to $\tau(x)$, and 
every 
point of $\tau(x)$ is in $\tau(f)$ for some edge $f$ incident with $x$,
 and
\item \label{stvi} every face $F$ can be associated with a vertical line segment $\tau(F)\subset 
\cc$ so that for every edge $e$ in the boundary of $F$, the square $\tau(e)$ is tangent to 
$\tau(F)$.
\end{enumerate}

It was shown in \cite{BeSchrHar} that a plane graph $G$ admits a square tiling exactly when $G$ is 
{\em uniquely absorbing}. We say that \g is \defi{uniquely absorbing}, if \fe\ finite subgraph 
$G_0$ 
\ti\ exactly one connected component $D$ of $\R^2\sm G_0$ which is \defi{absorbing}, that is, \rw\ 
on \g visits $G \sm D$ only finitely many times with positive probability (in particular, the 
subgraph of \g embedded in $D$ is 
transient, hence so is \G).

\subsection{Cutting the \rw\ flow along meridians} \label{sec disect}

A \defi{meridian} of \cc\ is a vertical line of the form $\{x\} \times \{0,1]\subset \cc$ for some 
$x\in \R / \Z$. Meridians are important, as they will allow us to `dissect' sub-flows of the \rw\ 
flow $i$.
To make this precise, given a vertex $x\in V(\G)$, we let $|x|$ denote the vertical `strip' of the 
cylinder \cc\ whose horizontal 
span coincides with that of the line segment $\tau(x)$ as described in \eqref{stv}: we let $|x|:= I 
\times \{0,a\} \subset \cc$, where $I$ is the interval of coordinates appearing in $\tau(x)$. Then 
$\tau(x)$ separates $|x|$ into two rectangles, and we denote the bottom one (that is, the one not 
meeting $\tau(o)$) by $\ceil{x}$. 

Next, we associate to this $x$ a flow $\check{x}$ from $x$ that `lives in  $\ceil{x}$'. Let us 
assume that each edge $e=vz$ of \G\ is directed `downwards', that is, the height coordinate of 
$\tau(v)$ is  higher than that of $\tau(z)$; we can make this assumption without loss of generality 
as we can always change the direction of an edge simultaneously with the sign of its flow. 
To define the flow $\check{x}$,  for every $e\in \arE(\G)$, let $\check{x}(e):=w(\tau(e) \cap 
\ceil{x})$ be the width of the rectangle  
$\tau(e)\cap \ceil{x}\subset \cc$ corresponding to $e$. (Thus if $\tau(e)$ is contained in 
$\ceil{x}$, then $\check{x}(e)=i(e)$ by \eqref{stii}, where $i$ is again the \rw\ 
flow  from $o$, and if $\ceil{x}$ dissects $\tau(e)$, then 
$\check{x}(e)<i(e)$.)  
A basic property of meridians (already observed in
\cite[Lemma 6.6]{planarPB}), is that $\check{x}$ is indeed a flow from $x$: to see this, let $v\neq 
x$ be any vertex such that $\tau(v)$ intersects $\ceil{x}$, and note that  $\check{x}$ brings flow 
into $v$ using the edges whose squares are tangent to $\tau(v)$ from above, and it removes flow into 
$v$ using the edges whose squares are tangent to $\tau(v)$ from below, and the total intensity of 
both these contributions equals the length of the intersection of $\tau(v)$ with $\ceil{x}$.

More generally, if $M,M'$ are two meridians intersecting $\tau(x)$, we let $\ceil{MxM'}$ denote the 
rectangle of \cc\ bounded by $M,\tau(x),M'$ and the bottom circle of \cc, and define the flow from 
$x$ 
that {\em lives in} $\ceil{MxM'}$ similarly to $\check{x}$, except that we replace the rectangle 
$\ceil{x}$ with $\ceil{MxM'}$ in the above definition.

The flows thus obtained always have finite energy, because the contribution of each edge $e$ to the energy is at most the area of $\tau(e)$ by the definitions, the whole area of \cc, and the interiors of $\tau(e), \tau(e')$ are disjoint for distinct edges $e,e'$.

\subsection{The basic lemma} \label{basic L}
The following lemma makes use of a square tiling to perform a certain `surgery' on the random walk 
flow $i$ on the \plg\ \Gd\ of a roundabout-transient graph \G. By recombining pieces of $i$ 
appropriately, we induce flows on $\Gr$ and $\Grd$ (or rather, on finite modifications of those 
graphs) that we will later use to make the intuition of \autoref{figfh} precise.

Every flow $i$ on \Gd\ induces a flow $i_\circ$ on $\Gr$, called the \defi{lift} of $i$ to $\Gr$, 
as follows. For every edge $e\in E(\Gd)$, we recall that $e$ is also an edge of \Gr, and just set 
$i_\circ(e)=i(e)$. For every other edge $e$ of \Gr, 
we let $i_\circ(e)$ be the unique value that forces $i_\circ$ to satisfy \knl\ at both 
endvertices $u,v$ of $e$. Such a value exists because $i$ satisfies \knl, and so the total 
divergence of  $u,v$ in $i_\circ$ is 0 for any value of $i_\circ(e)$.

\begin{lem}\label{4trans_graphs}
Let $G$ and $G^*$ be locally finite dual plane graphs. 
If $\mathring{G}$ is transient, then for some graph $H$ obtained from \g by contracting a finite 
connected subgraph into a vertex,
there are divergence free flows $f$ and $h$ of finite energy in $\mathring{H}$ and $\Hrd$ 
respectively, the supports of which intersect in a single edge (of 
$E(H)=E(H^*)$).
\end{lem}

The proof of this is a bit technical, but the main idea is quite simple. Let us assume that $H=G$ 
for a moment to explain the intuition. The interesting case is where $\mathring{G}$ is uniquely 
absorbing, in which case we can make use of the square tiling (of \Gd\ rather than \Gr\ for 
technical reasons). In this case, we use certain pairs of meridians to `dissect' four sub-flows 
$f_j$, from four distinct vertices $x_j$ to infinity, of the \rw\ flow on \Gd\ that live in four 
disjoint narrow rectangles of the tiling cylinder \cc\ of \Gd\ using the definitions of \autoref{sec 
disect}. Combining these flows in pairs using two finite flows, one from $x_1$ to $x_3$, and one 
from $x_2$ to $x_4$, we obtain two divergence free flows $f',h'$ in \Gd\ that `cross' in a manner 
corroborating the intuition of \autoref{figfh}. It is then straightforward to lift $f',h'$ to the 
desired flows $f,h$ in the two roundabout graphs using the above definition.

The statement of \autoref{4trans_graphs} may be a bit confusing, as it involves several graphs with 
shared edges. Our choice to work with \Gd\ may seem to be making matters worse at first sight, as it 
introduces one more graph. However, it makes life easier: rather than having to work with several 
graphs simultaneously, all non-trivial parts of the following proof deal with just one graph, \Gd. 
The nice aspect of \Gd\ is that it provides a concise representation of the graphs $G, G^*, \Gr$ and 
$\Grd$. The important property to remember is that the vertex set of \Gd\ is the (common) edge-set 
of $G$ and $G^*$, which is also the intersection of $E(\Gr)$ and $E(\Grd)$. Since the objective of 
\autoref{4trans_graphs} is a pair of divergence free flows in $\Gr$, $\Grd$ with a single common 
edge, this boils down to finding two divergence free flows in \Gd\ that cross at a single vertex. 
For technical reasons, it is a bit easier to find a pair of flows with finitely many crossings, and 
therefore we introduce the auxiliary graph $H$: after modifying a finite part of the graph where all 
crossings take place, it is easier to end up with a single crossing.

\begin{proof}[Proof of \autoref{4trans_graphs}]
We distinguish two cases, according to whether \Gd\ is uniquely absorbing or not.

If \Gd\ is uniquely absorbing, then  \cite{BeSchrHar} provides a square tiling of \Gd\ on a 
cylinder 
$\cc$ as described above, with $o$ being an arbitrary vertex of \Gd.

\begin{figure}
\centering \begin{overpic}[width=.5\linewidth]{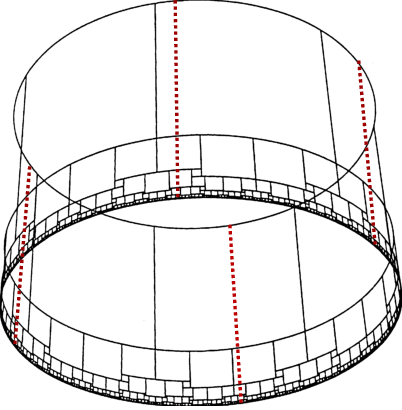} 

\end{overpic}
 \caption{\small An example of a square tiling, with the four meridians $M_j$  of 
\autoref{4trans_graphs} in dotted lines.} \label{figST}

\end{figure}

Our plan is to find four vertices $x_1, \ldots, x_4$ far enough from each other on \cc\ and flows 
$f_j$ from those vertices that live in appropriate disjoint rectangles, and combine these flows 
pairwise to obtain $f',h'$.
More precisely, we claim that we can choose four vertices $x_j, 1\leq j\leq 4$ in \Gd, a flow 
$f_j$ from each $x_j$, and a path $P_j$ from $x_j$ to $o$, so that these objects satisfy the 
following properties, which can be summarised by saying that these objects avoid to meet a common 
roundabout of \Gd\ 
whenever possible\footnote{Recall that \Gd\ is obtained from \Gr\  by contracting all edges outside 
roundabouts. Whenever we talk about a \defi{roundabout of \Gd} we will mean a roundabout of \Gr\ 
considered as a subgraph of \Gd.}.
\begin{enumerate}[A)]
\item \label{P0} $\supp(f_k) \cap \supp(f_j)=\emptyset$ for $k\neq j$; even stronger, no \rbt\ of 
\Gr\ meets both $\supp(f_k) $ and $ \supp(f_j)$;
\item \label{Pi} for every $j\leq 4$ and every edge $e$ of $P_j$, no edge of the roundabout of \Gr\ 
containing $e$ is in the support of any $f_k, 1\leq k \leq 4$, and
\item \label{Pii}  the roundabout of \Gr\ containing the first edge of $P_k$ does not contain $x_j$ 
and does not contain any edge of $P_j$ for $j\neq k$.
\end{enumerate}

Before proving that such a choice is possible, let us first see how it helps us construct the 
desired divergence free flows $f,h$. 

Let $X$ be the set of vertices $v$ of \g such that the roundabout $\mathring{v}$ in \Gd\ contains an 
edge of $P_j$ but does not contain the first edge of $P_j$. By construction, $X$ spans a connected 
subgraph of \G, since all  $P_j$ contain $o$. By modifying the $P_j$ appropriately if needed, we may assume that $X$ is a tree. Let $H$ be the graph obtained from $G$ by contracting 
$X$ into a single vertex $x$. 

It is straightforward to see that $\Hd$ can be obtained from \Gd\ by replacing all roundabouts 
corresponding to vertices in $X$ by the single roundabout $\mathring{x}$. The 
desired flows $f,h$ will be obtained as lifts ---as defined before the assertion of 
\autoref{4trans_graphs}--- of auxiliary flows $f',h'$ in \Hd\  constructed as follows. By the 
construction of $H$, the first edge of each $P_j$ lies in a roundabout $O_j$ that shares a vertex 
$y_j$ with $\mathring{x}$, and $O_j\neq O_k$ for $j\neq k$. In particular, $x_j$ lies on $O_j$ too; 
see \autoref{figRi}. Note that $O_j$ might have several vertices in common with $\mathring{x}$, 
because $H$ was obtained from $G$ by a contraction that may have introduced parallel edges. In this 
case, we choose $y_j$ so that $O_j$ contains an $x_j$--$y_j$ path $Q_j$ that only meets 
$\mathring{x}$ at $y_j$. 

Assume without loss of generality that $y_1,y_2,y_3,y_4$ appear in that order as we move around 
$\mathring{x}$ clockwise. We will construct a divergence free flow $f'$ as a linear combination of 
$f_1,f_3$, and a constant flow from $x_1$ to $x_3$ along a path $P_f$ contained in $O_1 \cup 
\mathring{x} \cup O_3$. We choose this $P_f$ to be a concatenation of $Q_1$, of one of the two 
$y_1$--$y_3$~paths contained in $\mathring{x}$, and of $Q_3$. 
To make the definition of $f'$ precise, suppose the
intensity\footnote{Recall that the {intensity} of $f_j$ is the divergence $f_j^*(x_j)$.} of $f_1$ 
is $\beta \in \R_+$, and the intensity of $f_3$ is $\gamma \in \R_+$. For each edge $e\in 
supp(f_1)$, we set $f'(e)=f_1(e)/\beta$. For each edge $e\in supp(f_3)$, we set 
$f'(e)=-f_3(e)/\gamma$. Finally, for each edge $e$ of $P_f$, we set $f'(e)=1$ if the direction of 
$e$ agrees with that of $P_f$ (which is from $x_1$ to $x_3$), and $f'(e)=-1$ otherwise. It is 
straightforward to check that $f'$ satisfies \knl.

\begin{figure}
\centering \begin{overpic}[width=.4\linewidth]{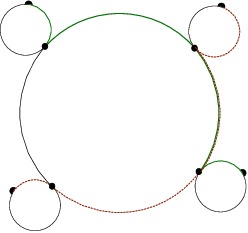} 
\put(70,70){$y_2$}
\put(20,70){$y_1$}
\put(57,24){$y=y_3$}
\put(20,24){$y_4$}
\put(1,95){$x_1$}
\put(80,95){$x_2$}
\put(89,19){$x_3$}
\put(5,12){$x_4$}
\put(40,93){$P_f$}
\put(44,0){$P_h$}
\put(44,44){$\mathring{x}$}
\end{overpic}
 \caption{\small The roundabouts  $O_1,O_2,O_3,O_4$ and $\mathring{x}$, along with the 
$x_1$--$x_3$~path $P_f$ (shown in green, if colour is shown) and the $x_2$--$x_4$~path $P_h$ 
(dashed, red) used in the definition of $f',h'$, respectively. } \label{figRi}

\end{figure}
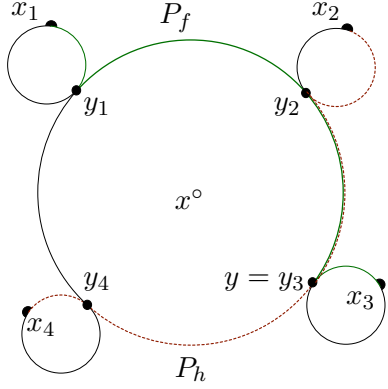

Similarly, we construct the flow $h'$ as a linear combination of $f_2,f_4$, and a constant flow from 
$x_2$ to $x_4$ along a path $P_h$ contained in $O_2 \cup \mathring{x} \cup O_4$, obtained by 
concatenating $Q_2$ and $Q_4$ with one of the two $y_2$--$y_4$~paths contained in $\mathring{x}$. 
Finally, let $f$ be the lift of $f'$ to $\Hr$ and let $h$ be the lift of $h'$ to $\Hrd$. 

We claim that these flows satisfy our requirement $|supp(f) \cap supp(h)|=1$. To see this, we 
observe that  there is a
unique  vertex $y$ of $\mathring{x}$ at which $P_f$ switches between two roundabouts of $\Hr$ and 
simultaneously $P_h$  switches between two roundabouts 
of $\Hrd$.\footnote{In the example of \autoref{figRi}, we have $y=y_3$. If we had lifted $f'$ to 
$\Hrd$ and $h'$ to $\Hr$ instead, then we would have had $y=y_2$. If we had chosen a $P_f$ that 
uses 
the other $y_1$--$y_3$~path of $\mathring{x}$, then we would have had $y=y_4$.} Indeed, $P_f$ stays 
within a roundabout of $\Hr$ except precisely at the vertices $y_1$ and $y_3$, where it switches 
from $O_1$ to $\mathring{x}$ and from $\mathring{x}$ to $O_3$ respectively. Moreover, $P_h$ 
contains exactly one vertex $y\in \{y_1,y_3\}$, and it contains two edges of  $\mathring{x}$ 
incident with $y$, therefore it switches between the two roundabouts of $\Hrd$ incident with $y$.

We claim that the unique edge (of $E(H)=E(H^*)$) in $supp(f) \cap supp(h)$ incident with $x$  is the edge corresponding to $y$. To see this, note first that $supp(f_i)$ avoids all edges  incident with $x$  by \eqref{Pi}, and so it remains to check our claim for the part of $f$ and $h$ arising as lifts of the unit flows we sent along $P_f$ and $P_h$. Since $P_f$ stays 
within a roundabout of $\Hr$ except precisely at $y_1$ and $y_3$, by the definition of a lift 
we deduce that the only edges  in $supp(f)$ incident with $x$  are the edges corresponding to $y_1$ and $y_3$. Since $P_h$ contains exactly one of these vertices $y$, we deduce that $supp(h)$ contains the corresponding edge, but does not contain the other edge  in $supp(f)$ incident with $x$. This proves our claim.

Finally, no edge that is not incident with $x$ can lie in $|supp(f) \cap supp(h)|$ by properties 
\eqref{P0}--\eqref{Pii}: these properties were designed exactly so as to prevent further 
intersections.

\medskip
Thus, in the uniquely absorbing case, it only remains to prove that we can indeed choose vertices 
$x_j$, flows $f_j$, and paths $P_j$ with properties \eqref{P0}, \eqref{Pi} and \eqref{Pii} above.

For this, recall that the length of the circumference of \cc\ is 1, and  let $M_j, 0\leq j < 4$ 
denote 
the meridian of \cc\ whose width coordinate is $j/4$. 
Let $S$ be the set of roundabouts $O$ that contain an edge $e$ with $w(\tau(e))\geq 1/8$. As the set $S$  is finite, we may let $b>0$ be the least vertical coordinate in the set $\bigcup_{O\in S} \tau(O)$, where $\tau[O]:= \bigcup_{e\in E(O)} \tau(e)$ comprises the squares of the edges of $O$. For each $j$, pick $h_j< \min(b,1/16)$.
In addition, we choose $h_j$ even smaller, if needed, to ensure that 
if 
$x$ is a vertex such that $\tau(x)$ meets $M_j$ below height $h_j$, then $w(\tau(x))< 1/8$; this is 
possible because there are only finitely many edges $e$ with $w(\tau(e))$ greater than any fixed 
constant since $\cc$ has finite area, and any horizontal line segment meets $\tau(x)$ in at most three squares by \eqref{stv} and the fact that \Gd\ is 4-regular.

Let $\ceil{h_j M_j}$ denote the sub-interval of $M_j$ with height coordinates ranging between zero 
and 
$h_j$, and $\floor{h_j M_j}$ the sub-interval of $M_j$ with height coordinates ranging between $h_j$ 
and 
1.

It is proved in \cite[Theorem~4.1~(v)]{BeSchrST} that almost every meridian with respect to Lebesgue 
measure meets only finitely many squares of the tiling lying above any fixed height. We may assume 
that our $M_j, 0\leq j <4$, all have this property, for otherwise we can achieve it by rotating \cc.
Therefore, for every $j< 4$, there is a lowest square  $\tau(e_j)$ meeting $\ceil{h_j M_j}$ such 
that the \rbt\ 
$O_j$ of \Gr\ containing the edge $e_j$ also contains an edge $g_j$ meeting $\floor{h_j M_j}$ 
(\autoref{figOi}); this is true because $\floor{h_j M_j}$ only meets finitely many 
squares of positive area, and so there are finitely many \rbt s to choose from. There is at least 
one to choose from: a \rbt\ whose image contains the point of $M_j$ at height $h_j$.

\begin{figure}
\centering \begin{overpic}[width=.15\linewidth]{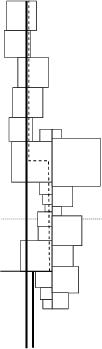} 
\put(30,36){$h_j$}
\put(2,-5){$M_j$}
\put(-20,6){$\ceil{M_j x_j M_j'}$}
\put(10,-4){$M_j'$}
\put(-5,22){$x_j$}
\put(9,25){$e_j$}
\put(9,50){$g_j$}
\put(16,55){$P_j$}
\end{overpic}
\vspace*{.3cm}
 \caption{\small The choice of $x_j, f_j$ and $P_j$.} \label{figOi}

\end{figure}
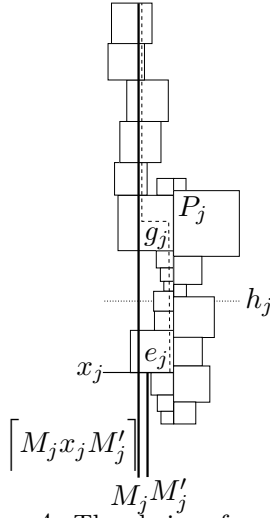

Let $x_j$ denote the endvertex of $e_j$ whose height coordinate is lower, and note that $\tau(x_j)$ 
meets $M_j$. Let $M_j'$ be a meridian meeting $\tau(e_j)$ (and in particular $\tau(x_j)$) close 
enough to $M_j$, but distinct from $M_j$, that the rectangle $\ceil{M_j x_j M_j'}$ bounded by $M_j, 
\tau(x_j), M_j'$ and the bottom circle of \cc, meets the $\tau$ image of no \rbt\ meeting $\floor{h_j 
M_j}$; such a $M_j'$ exists because, by the choice of $e_j, O_j$, no roundabout meeting $\lfloor h_j
M_j \rfloor$ has an edge $e$ such that $\tau(e)$ meets $M_j$ below
$\tau(x_j)$, or we would have chosen $e$ instead of $e_j$. 

As we 
can choose $M_j'$ as close to $M_j$ as we wish, we may assume that $d(M_j,M_j')<1/16$, which will 
be 
useful later.

Let $f_j$ be the flow from $x_j$ that lives in $\ceil{M_j x_j M_j'}$, as defined in 
\autoref{sec disect}. Recall that $f_j$ must have finite energy. 
We claim that 
\labtequ{w8}{If $e\in \supp(f_j)$, then $\tau(e)$ is contained in the open vertical strip of radius 
$1/8$ centered at $M_j$.}
Indeed, by the definition of $f_j$, if $e\in \supp(f_j)$, then $\tau(e)$ intersects the interior of 
$\ceil{M_j x_j M_j'}$. Then $\tau(e)$ cannot have a point at height higher than $h_j$, which we 
recall is less than  $1/16$, because it would have to intersect the interior of $\tau(e_j)$ in that 
case, contradicting \eqref{stiii}. Thus the height of $\tau(e)$ is at most $1/16$, and being a 
square, so is its width. Together with our assumption that $d(M_j,M_j')<1/16$, this proves our 
claim.

Note that \eqref{w8}, combined with the choice of the $M_j$, immediately implies that $\supp(f_k) 
\cap \supp(f_j)=\emptyset$ for $k\neq j$; in fact, it even implies the stronger statement of 
\eqref{P0}, because by \eqref{stvi} if edges $e,f$ lie in a common \rbt\ then $\tau(e),\tau(f)$ must 
meet a common meridian.

It remains to construct the paths $P_j$: we let $P_j$ start with the $x_j$-$g_j$ path in $O_j$ 
containing $e_j$, and continue with the $g_j$-$o$ path consisting of all the edges whose 
$\tau$-image meets $M_j$ above $\tau(g_j)$. Recall that there are only finitely many such squares as 
we remarked above. The fact that  
the edges whose $\tau$-image meets $M_j$ above $\tau(g_j$) form a $g_j$-$o$ path follows from 
\eqref{stv} and the fact that $\tau(o)$ is the top circle of $\cc$. In fact, by the above argument, 
we can even assume that $M_j$ does not contain a boundary of any square $\tau(e)$, and so $M_j$ 
uniquely determines that $g_j$-$o$ path. Note that by construction, 
\labtequ{Piedge}{every edge of $P_j$ is in a \rbt\ $O$ such that $\tau[O]$ meets $M_j$.}

To see that \eqref{Pi} is satisfied, note that if $O$ is any roundabout containing an edge in the 
support of $f_j$, then $O$ cannot contain any edge in any of the $P_k$. This is true for $k = j$ by  
the definition of $M_j'$ (see \autoref{figOi}). For $k \neq j$, if $e$ is in the support of $f_j$ 
then $\tau(e)$ cannot have a point at height higher than $h_j$. As we chose $h_j < b$, all the other 
edges $e'$ in the roundabout containing $e$ have $w(\tau(e')) < 1/8$. Thus  \eqref{Pi} follows from 
\eqref{w8} and \eqref{Piedge}.

Finally, we can prove \eqref{Pii} by a similar argument, now using the fact that $w(\tau(x_j))< 
1/8$ 
by the second part of our definition of $h_j$, and the fact that the \rbt\ containing the first 
edge 
$e_j$ of $P_j$ cannot have any squares of side length $1/8$ or greater, and therefore $\tau[O]$ cannot intersect $M_k$ for any $k\neq j$.

Thus all three desired properties \eqref{P0}--\eqref{Pii} are satisfied, and as discussed above this 
completes the case where 
\Gd\ is uniquely absorbing.

\medskip
Suppose now \Gd\ is not uniquely absorbing. 
 Then for some finite subgraph $G_0$ of \Gd, we have at least two absorbing components $D_1,D_2$ in 
$\R^2\sm G_0$. By elementary topological arguments, $G_0$ contains a cycle $C$ such that both the 
interior $I$ and the exterior $O$ of $C$ contain transient subgraphs of \Gd, namely one of its face 
boundaries.

If any of these subgraphs $I,O$ is uniquely absorbing, then we can repeat the above arguments to 
that subgraph to obtain the two desired flows. 

Hence it remains to consider the case where there is a cycle $C_I$ in $I$ and a cycle $C_O$ in $O$ 
that further separate each of $I,O$ into two transient sides. In fact, we can iterate this argument 
as often as we like, to obtain many distinct transient subgraphs separated from any given cycle. Let 
us iterate it often enough to obtain four disjoint cycles $C_j, 1\leq j \leq 4$, and inside each $C_j$ a 
cycle $D_j$ such that the interior of $D_j$ is transient and no roundabout of \Gr\ meets any two of 
these eight cycles. We remark that the $D_j$ can be chosen internally
disjoint even if some or all of the $C_j$ are concentric. 

We now apply \autoref{lyons} to each of the four interior sides of the $D_j$ to obtain four flows of 
finite energy  
$f_j$ from vertices $x_j$, such that the support of $f_j$ is contained in $D_j$. We can then 
combine those flows pairwise in a way similar to the uniquely absorbing case to obtain the two 
desired flows $f,h$: we can let $o$ be an arbitrary vertex outside all $C_j$, 
pick paths $P_j$ from $x_j$ to $o$, and again consider a graph $H$ obtained from \g by contracting 
the vertices corresponding to all roundabouts meeting the $P_j$ except for the first one. We then 
construct $f',h'$, and from them $f,h$, as indicated in \autoref{figRi}. The fact that $|\supp(f) 
\cap \supp(h)|=1$ follows 
from the same graph-theoretic arguments about the structure of \Gd, for which we did not need the 
square tiling.
\end{proof}

\section{Harmonic functions on plane graphs} \label{sec Hfpg}

In this section, we use \autoref{flow_pot_char} to prove a new existence criterion for 
non-constant Dirichlet harmonic functions in planar graphs, \autoref{flow_flow_char_roundabout} 
below, which is used 
in the proof of \autoref{thm:UK-trans}.
Before proving \autoref{flow_flow_char_roundabout}, we prove the following which may be of 
independent interest, and can be thought of as a warm-up towards the harder 
\autoref{flow_flow_char_roundabout}. The reader will lose nothing by skipping directly to 
\autoref{flow_flow_char_roundabout}. 

\begin{thm}\label{flow_flow_char}
 Let $G$ and $G^*$ be locally finite 1-ended dual plane graphs.
Then the following are equivalent:
\begin{enumerate}[A)]
 \item \label{ffc a} $G\not \in {\Ocal}_{HD}$;
\item \label{ffc b}  $G^*\not \in {\Ocal}_{HD}$;
\item \label{ffc c} there are divergence free flows $f$ and $h$ of finite energy in $G$ and $G^*$, 
respectively, whose supports 
intersect in a single edge.
\end{enumerate}
\end{thm}

\begin{proof}
By symmetry, it suffices to show that \eqref{ffc a} is equivalent to \eqref{ffc c}. 
For this,  assume first that $G\not \in {\Ocal}_{HD}$. Then by \autoref{flow_pot_char_reform} \g 
admits a 
divergence free flow $f$ and a potential $\rho$ such that both $f$ and 
$\partial \rho$ have finite energy and their supports intersect in a 
single edge. 
As $\partial \rho$ satisfies 
Kirchhoff's cycle law in $G$, when considered as a function on the dual  $G^*$ it satisfies 
Kirchhoff's node law at every vertex; that is, $\partial \rho$ is a divergence free flow of $G^*$. 
Hence $f$ and $h:=\partial \rho$ satisfy \eqref{ffc c}. 

For the converse, suppose \eqref{ffc c} holds.
Consider $h$ as a function on the edges of $G$. We are going to apply 
\autoref{potential_or_harmonic} to $G^*$ to
deduce that  $h$ satisfies Kirchhoff's cycle law in  $G$. Since $G$ is one-ended, item \eqref{ph b} 
of
\autoref{potential_or_harmonic}  cannot be satisfied, hence item \eqref{ph a} applies and says that  
$h$ satisfies Kirchhoff's cycle law in $G$.
Thus by \autoref{kcl pot} there is a potential $\rho$ in $G$ with $\partial \rho = h$, and so by 
\autoref{flow_pot_char_reform} the flow $f$ and the potential $\rho$ witness that
$G\not \in {\Ocal}_{HD}$. 
\end{proof}

\begin{eg}
We give a simple example of a graph $G$ such that neither the second nor the third condition imply 
the first  in \autoref{flow_flow_char} if we omit the assumption that $G$ and $G^*$ are 
1-ended.
We first construct an auxiliary graph $H$ from the disjoint union of a family of cycles $C_n, n\in 
\Nbb$, where $C_n$ has length $2^n$, by gluing $C_n$ and 
$C_{n+1}$ together along an edge for each $n\geq 2$; we choose the two gluing edges in $C_n$ so 
that they have distance $|C_n|/2-1$. 
We obtain the graph $G$ by attaching two copies of $H$ at distinct vertices of a triangle $T$.
Clearly, the graph $G$ is in ${\Ocal}_{HD}$. 
In \autoref{fig:cycle_counterX} we will construct an embedding of the graph $G$ such that the second 
and 
third condition are satisfied. 

To see this, we consider the embedding of $G$ in the plane indicated in 
\autoref{fig:cycle_counterX}. The dual $G^*$ corresponding to this embedding is a 1-way infinite 
path with many parallel edges; in fact the removal of any vertex splits it into two transient 
subgraphs.  Easily, $G^*$ has a Dirichlet harmonic function (see e.g.\ 
\cite[Theorem~4.20]{SoardiBook}). To see that the third condition is satisfied, we let $f$ be a 
divergence free flow in $G$ supported on the triangle $T$, and let $h$ be a flow with infinite 
support in $G^*$ that uses only one edge of $T$; the latter exists because the intersection of each 
side of $T$ with $G^*$ is transient.
\end{eg}
   \begin{figure} [htpb]   
\begin{center}
       \includegraphics[height=4cm]{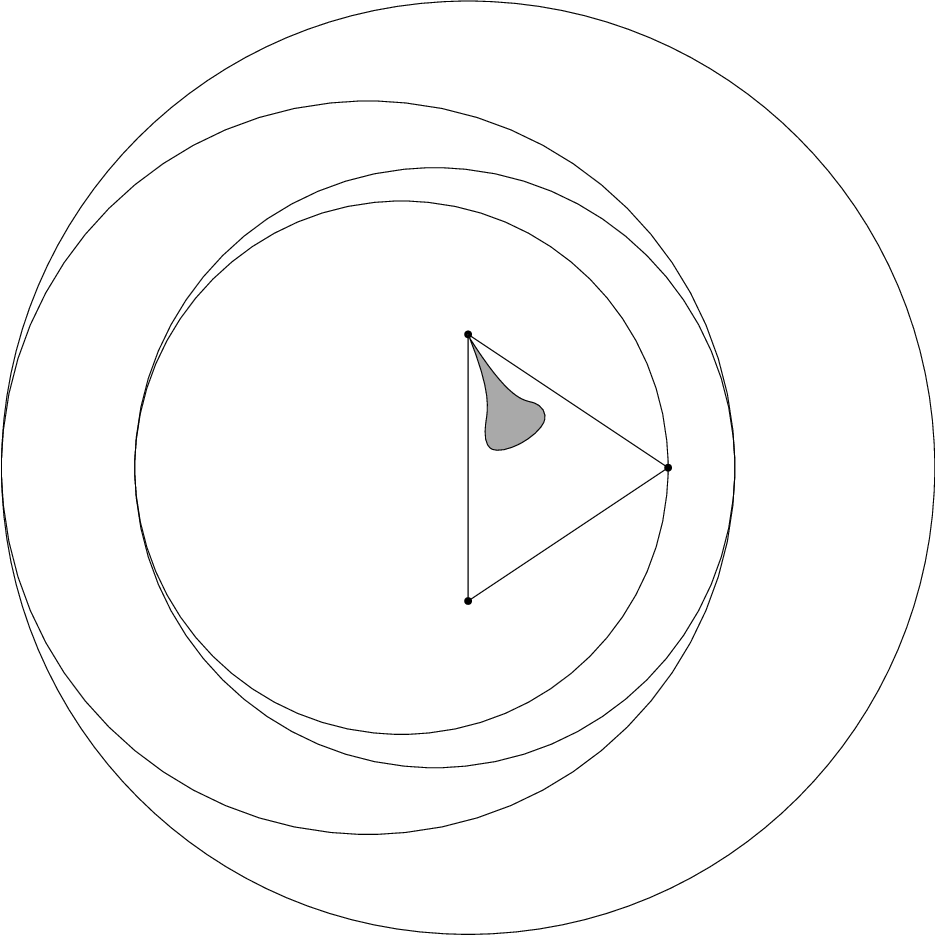}
         \caption{ An embedding of the graph $G$ in the plane. One copy of $H$ is embedded on the 
outside of the triangle. The other copy is embedded in the grey region in an analogous way (here we 
embed the cycle $C_{n+1}$ inside $C_n$).}\label{fig:cycle_counterX}
\end{center}
   \end{figure}

The next result provides a strengthening of condition \eqref{ffc c} of \autoref{flow_flow_char} 
which implies that $G\not \in {\Ocal}_{HD}$ even if $G$ has more than one end. 

\begin{thm}\label{flow_flow_char_roundabout}
 Let $G$ and $G^*$ be locally finite dual plane graphs such that  their {roundabout graphs} 
$\mathring{G}$ and $\Grd$ admit divergence free flows $f$ and $h$ respectively, both of finite 
energy,
the supports of which intersect in a single edge (of 
$E(G)=E(G^*)$).
Then $G\not \in {\Ocal}_{HD}$.
\end{thm}

\begin{proof}
Since divergence free flows satisfy Kirchhoff's node law at finite vertex-sets, the restriction of 
the flow $h$ of 
$\Grd$ to the edges of $G^*$ is a divergence free flow in $G^*$. We denote that flow by 
$h_{G^*}$. We distinguish two cases. 

{\bf Case 1: the flow $h_{G^*}$ considered as a function on the edges of $G$ 
satisfies \kcl\ in $G$.} 

Then  $h_{G^*}= \partial \rho$ for some potential $\rho$ on $G$ by \autoref{kcl pot}. As above, the 
restriction of $f$  to the edges of $G$ is a divergence free flow $f_G$ in that graph.  
Then $f_G$ and the potential $\rho$ of $G$ witness that 
$G\not \in {\Ocal}_{HD}$ by \autoref{flow_pot_char_reform}.

\smallskip
Having dealt with Case 1, by \autoref{potential_or_harmonic} (applied to $G^*$) it remains to 
consider the following.

{\bf Case 2:  there is a finite cut $c$ of $G^*$ such that $h_{G^*}$ witnesses that at least two 
components  of $G^*-c$ are transient.}\\ 
We start with a slightly technical argument that essentially shows that it suffices to consider the 
case that the cut $c$ is a bond. 
Let $\tilde D_1$ and $\tilde D_2$ be components of $G^*- c$ such that $h_{G^*}$ witnesses that they 
are transient. Let $b$ be a minimal cut contained in the cut $c$ that separates some vertex of 
$\tilde D_1$ from some vertex in $\tilde D_2$.  
Let $D_i$ be the component of $G^*- b$ including $\tilde D_i$ (for $i=1,2$). 
By setting $h_{G^*}$ equal to zero at components of $G^*- 
c$ different from $\tilde D_1$ and $\tilde D_2$, and by multiplying all its values in $\tilde D_1$ 
by the same constant if necessary, we may assume and we do assume that $h_{G^*}$ witnesses that 
$D_1$ and $D_2$ are transient. 

Having finished this slightly technical part, we conclude that the bond $b$ considered as an edge 
set of $G$ is the set of edges of a cycle $C$, such that $D_1$ and 
$D_2$ lie on different sides of $C$  by \autoref{lf bonds}, see \autoref{fig:C}.

      \begin{figure}[htpb]
\begin{center}
   	  \includegraphics[height=3 cm]{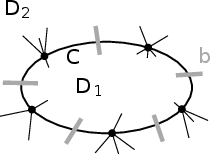} 
   	  \caption{The bond $b$ in $G^*$, drawn grey, separates the components $D_1$ and $D_2$.
   	  The corresponding cycle $C$ in \G, drawn thick, separates two transient subgraphs 
associated to the components $D_1$ and $D_2$.}\label{fig:C}
\end{center}
   \end{figure}

Our plan is to show that the two subgraphs $G_1,G_2$ of \g in either side of $C$ ---defined more 
formally below--- are transient, and apply \autoref{fin_sep} to deduce that $\g \not\in 
{\Ocal}_{HD}$. Since we know that $D_1,D_2$ are transient subgraphs of $G^*$, we would like to pass 
this information to the dual $G$ to deduce that  $G_1,G_2$ are transient too. The tool we have is 
\autoref{dual trans}, but there are two difficulties in applying it: firstly, we need $D_1,D_2$ to 
be roundabout-transient rather than just transient to apply this tool. Secondly, $D_i$ is not quite 
the dual of $G_i$, as the dual of a subgraph is not quite a subgraph of the dual.

To overcome the first difficulty, recall that every cut of $G^*$ is a cut of $\Grd$ by the 
definitions, and so we can think of $b$ as a cut of $\Grd$.
Recall moreover that $h_{G^*}$ was obtained from $h$ by restriction. But since $h_{G^*}$ witnesses 
that both components $D_1,D_2$ of $G^*-b$ are transient, it follows from the definitions that $h$ 
witnesses that both components $D'_1,D'_2$ of $\Grd-b$ are transient. In other words, $D_1,D_2$ are 
both roundabout-transient. (Indeed, $D_i'$ is almost equal to the roundabout graph $D_i^\circ$ of 
$D_i$; that is, they agree except at the finitely many roundabouts that contain endvertices of the 
bond $b$. However, changing finitely many vertices does not affect transience\footnote{Formally, 
we can argue similarly, as in the `second difficulty' explained below.}). Hence their duals are 
transient by \autoref{dual trans}. 

It remains to overcome the second difficulty, namely to explain the relationship between $D_i^*$ and 
$G_i$, where we define $G_1$ to be the subgraph of \g spanned by all vertices lying on the cycle $C$ 
and its inside, and we define
$G_2$ to be the subgraph of \g spanned by all vertices lying on the cycle $C$ and its outside. Let 
$G_i'$ be the graph obtained from $G_i$ by contracting  $C$ into a single vertex (we may create 
parallel edges by this contraction, but this is ok). 

By the definition of the dual of a plane graph, deleting an edge in the primal corresponds to 
contracting the same edge in the dual, and vice-versa \cite{Oxley2}. This is still true when the 
deleted edges disconnect the graph into two components $C_1,C_2$; in this case, the  corresponding 
contractions in the dual create a cutvertex $v$, disconnecting it into two components $C'_1,C'_2$ 
and the dual of each $C_i$ coincides with the graph spanned by $C'_1$ and $v$. Applying this fact in 
our situation, we observe that  $D_i^*$ coincides with $G_i'$, because $D_1 \cup D_2$ is obtained 
from $G^*$ by deleting the edges in $b$, and so the dual of $D_1 \cup D_2$ is the graph obtained 
from \g by contracting the edges in $C$.

To summarise, we have proved that $G_1', G_2'$ are transient. Hence so are the subgraphs $G_1'', 
G_2''$ of $G$ obtained by deleting the contracted vertex from each of $G_1', G_2'$ (in other words, 
the subgraphs of \g lying in either side of $C$). Applying \autoref{fin_sep} to these subgraphs, we 
deduce that $\g \not\in {\Ocal}_{HD}$ (to be more precise, we apply \autoref{fin_sep} to $G_1'', G_2 
(= G_2''\cup C)$ to make sure these subgraphs define a cut of $G$, i.e.\ they bipartition $V(G)$, 
but as transience is preserved by finite modifications, this is straightforward).

\end{proof}

\section{Proof of the main result} \label{sec proof}

We can now prove \autoref{thm:UK-trans}.
\begin{proof}
We have already collected enough tools for the case where $G^*$ is locally finite too: in this case, 
we can apply \autoref{4trans_graphs} to deduce that for some graph $H$ obtained from \g by 
contracting a finite connected subgraph, there are divergence free flows $f$ and $h$ in $\Hr, \Hrd$ 
respectively intersecting at a single edge. 
Plugging this into \autoref{flow_flow_char_roundabout} 
we deduce that $H\not\in {\Ocal}_{HD}$. Since $H$ differs from \g in finitely many vertices and edges, we easily obtain ---e.g.\ using \autoref{flow_pot_char}--- that $G\not\in {\Ocal}_{HD}$ as claimed.

Thus it remains to consider the case where $G^*$ is not locally finite, or in other words, where \g has faces bounded by infinitely many edges. We will reduce this case to the above, by constructing a supergraph $T$ of $G$ with locally finite dual $T^*$ such that $G\in {\Ocal}_{HD}$ if and only if $T\in {\Ocal}_{HD}$.

For this, let us first construct a supergraph  $G'$ of $G$ obtained by 
adding edges in order to split every infinite face of $G$ into finite faces in such a way that each vertex of $G$ receives at most 2 
new edges per incident 
face 
(any finite number would do in place of 2). This is easy to 
do recursively by enumerating the vertices of $G$ that lie on an infinite face, and in each step $i$ adding a finite set of edges $C_i$, disjoint from all $C_j, j<i$, that puts the next available vertex in the enumeration on a finite face boundary. 
As $V(G)$ is countable, so is the set of newly added edges. Fix an enumeration $(e_n)_{n\in \Nbb}$ of the set of newly added edges, and subdivide $e_n$ by $2^n$ new vertices. Let  $T$ denote the resulting graph.

Note that $T$ is locally finite, and all its face boundaries are finite, hence $T^*$ is locally 
finite. Its roundabout graph $T^\circ$ has a subgraph $T'$ which can be obtained from \Gr\ by 
subdividing each edge at most twice: we obtain $T'$ by deleting from 
$T^\circ$ the roundabouts corresponding to the vertices in $V(T) \sm V(G)$; the subdivisions are due 
to the newly added edges $e_n$. By \autoref{lyons}, $T^\circ$ is transient since \Gr\ is. 
As $T^*$ is locally finite, we can prove that $T \not\in {\Ocal}_{HD}$ by the arguments of the first paragraph of this proof. 

We now claim that $T \not\in {\Ocal}_{HD}$ implies the  desired $G\not\in{\Ocal}_{HD}$. Indeed, this follows from Corollary~1.2 of \cite{C:harmonic}, which states that if a connected graph $G$ is obtained from a connected graph $T$ by deleting a set of edges of finite total conductance, then $T \in {\Ocal}_{HD}$ if and only if $G\in{\Ocal}_{HD}$. In our setup all edges have conductance 1, but we can replace each path of length $2^n$ that we attached to $G$ to obtain $T$ by a single edge of conductance $1/2^n$; by the classical series law (see e.g.\ \cite[{Section~2.3}]{LyonsBook}), this modification results in a network $T'$ that is `equivalent' to $T$, in particular, $T'\in {\Ocal}_{HD}$ if and only if $T\in{\Ocal}_{HD}$. As the sum of these conductances is finite, the aforementioned result applies, and we deduce that $G\not\in{\Ocal}_{HD}$.
\end{proof}

\section{Non-amenable graphs} \label{secApps}

A vertex is in the \emph{neighbourhood $\partial X$} of some vertex set $X$ if it is not in $X$ but 
shares an edge with a vertex in $X$.\footnote{With a slight abuse of notation we use the operator 
$\partial$ to denote two unrelated concepts: the difference operator of a potential, as well as the 
set of neighbours of vertex-sets in the context of non-amenability.} 
An infinite graph $G$ is \emph{non-amenable} if there is a constant $\gamma>0$ such that for
every finite vertex set $S$ of $G$ we have $|\partial S| \geq \gamma\cdot |S|$. 
For a nonempty vertex-set $X$, we let $ch({X})=\frac{|\partial {X}|}{|{X}|}$, and define the 
the \emph{Cheeger-constant}  $ch(G)$ of a graph $G$ to be the infinimum of $ch({X})$ ranging over 
all finite nonempty vertex-sets.

\begin{lem}\label{non-amenable_roundabout}
If a (simple) locally finite plane graph $G$ is non-amenable, then so is its roundabout graph 
$\mathring{G}$. 
\end{lem}

\begin{proof}
Let $X$ be a finite vertex set of $\mathring{G}$. Let $\overline X$ be the set of those vertices of 
$G$ whose roundabouts contain vertices of $X$. 

We need to show that $|\partial X| \geq \gamma |X|$ for some $\gamma>0$. 
The next claim will imply this under the assumption that  $X$ is much larger than $\overline X$:
\labtequ{less_6}{Less than $6\cdot |\overline X|$ vertices of $X$ have all their neighbours in $X$.}
To prove this, let $Y$ be the set of those vertices of $X$ with all their neighbours in $X$.
If $v\in Y$, then the unique vertex of $\mathring{G}$ that shares an edge of $G$ with $v$ lies in 
$X$. Thus $|Y|\leq 2\cdot |E(\overline X)|$, where $E(\overline X)$ denotes the set of edges of 
${G}$ with both end-vertices in $\overline X$. As the subgraph $(\overline X, E(\overline X))$  of 
${G}$ spanned by $\overline X$ is planar, it has average 
degree less than 6, and so 
$|E(\overline X)|< 3\cdot |\overline X|$. Thus $|Y|< 6\cdot |\overline X|$ as claimed.
\medskip

Now if $|X|\geq 12\cdot |\overline X|$, then by \eqref{less_6}, at least $|X|/2$ vertices of $X$ 
have 
a neighbour outside 
$X$. As $\mathring{G}$ has maximum degree three, $\partial X$ then has size at least $|X|/6$, which 
fulfils our aim with $\gamma=1/6$. 

Hence it suffices to consider sets $X$ with $|X|< 12\cdot |\overline X|$, and we will assume this is 
true from now on. 

It is reasonable to expect that non-amenability is most difficult to prove when the set 
$X$ is a union of roundabouts. With this intuition in mind, it is natural to consider the 
following set. 
Let $\overline{\overline X}$ be the set of those vertices of $\overline X$, the whole roundabout of 
which
is 
in $X$. 
Let $\epsilon$ be the proportion of the 
remaining vertices of $\overline X$, that is, $\epsilon:= ( 
|\overline X|-|\overline{\overline X}|)/|\overline X|$.  Our next claim is
\labtequ{epsilon}{
 $|\partial X|> \frac{\epsilon}{12} |X|$.}
To see this, note that the roundabout $\mathring{x}$ of each $x\in \overline X\sm 
\overline{\overline X}$ contains a distinct vertex of $\partial X$, namely, a vertex contained in 
$\mathring{x}$ but not in $X$, hence $|\partial X|\geq |\overline X\sm \overline{\overline X}|= 
\epsilon \cdot |\overline X|$. Thus the 
claim follows from our assumption that $|X|< 12\cdot |\overline X|$. 
\medskip

If $\epsilon$ is bounded below, then \eqref{epsilon} says that $\mathring{G}$ is non-amenable.
Our next claim will help deal with the case where $\epsilon$ is small.
\labtequ{delta}{
$|\partial X|\geq K(\epsilon) \cdot  |X|$, where $K(\epsilon)=\frac{ch(G)\cdot (1-\epsilon)- 
\epsilon}{12}$.}
Indeed, a lower bound for the neighbourhood $\partial X$ of $X$ is the cardinality of the  set 
$\overline N$ of roundabouts 
containing vertices of $\partial X$.
Clearly, a vertex $x$ of the neighbourhood $\partial \overline{\overline{X}}$ of 
$\overline{\overline{X}}$ is in $\overline N$ unless it is in $\overline {X}$. As $x$ cannot be in 
$\overline{\overline{X}}$ we can strengthen this statement 
slightly by replacing $\overline {X}$ by $\overline {X}\setminus 
\overline{\overline{X}}$. Putting these observations together, we have
\[
 |\partial X|\geq |\overline N| \geq |\partial  \overline {\overline X}|- |\overline X\sm \overline 
{\overline X}|
\]
\[
\ \ \ \ \ \ \ \ \ \ \ \ \geq ch(G) \cdot  |\overline {\overline X}|- \epsilon |\overline X| 
\]
Note that $|\overline{\overline X}|=(1-\epsilon) \cdot |\overline X|$ by the definition of 
$\epsilon$. Since we are assuming that $ |\overline X| > |X|/12$, we obtain the desired $|\partial 
X|\geq \frac{ch(G)(1-\epsilon)-\epsilon}{12} \cdot  |X|$. 
\medskip

Combining \eqref{delta} with \eqref{epsilon} it is straightforward to check that $|\partial 
X|\geq\gamma |X| $ for some $\gamma>0$ depending on $ch(G)$. Thus $\mathring{G}$ is non-amenable.
\end{proof}

We can now prove one of the main results mentioned in the introduction.

\begin{proof}[Proof of \autoref{nonam_UK-trans_intro}]
If $G$ is non-amenable, then so is \Gr\ by \autoref{non-amenable_roundabout}.  Every non-amenable 
locally finite graph is transient as it contains a subtree with positive Cheeger-constant by a 
result of Benjamini and Schramm \cite{BeSchrChee}, and applying this fact to \Gr\ proves the statement.
\end{proof}

\begin{rem}
 The non-amenability condition in \autoref{nonam_UK-trans_intro} cannot be relaxed into 
the weaker  {\em anchored vertex expansion}. Here we say that $G$ has  {\em 
anchored vertex expansion}, if there 
is a constant $\gamma>0$ such that for
every finite connected vertex set $S$ of $G$ containing a fixed vertex $o$, we have $|\partial S| 
\geq \gamma\cdot |S|$. (That is, we modify the definition of non-amenability by just imposing the 
condition $o\in S$ and connectedness.) This is shown by the following example. 
\end{rem}

\begin{eg}\label{eg:anchor}
We construct a plane tree with non-zero anchored vertex expansion whose roundabout 
graph has zero anchored vertex expansion. We start with a ray whose 
vertices are labelled by the non-negative integers. For each squared number $n^2$, we attach 
a large tree at the vertex with that label. More precisely, at the vertex labelled $n^2$ we 
attach $(n+1)^2$ new neighbours, and at each of them we attach a full binary tree. 

We embed this graph in the plane as indicated in \autoref{fig:anchor}. The only property of 
this embedding we are using is that there is a face whose boundary contains  
the original ray as a subpath. 

It is straightforward to check that this tree $T$ has non-zero anchored vertex expansion but $T^\circ$ contains facial paths of length $n^2$ that have only $2n$ neighbours. Hence the  anchored vertex expansion of $T^\circ$ is zero. 

   \begin{figure}[htpb]
\begin{center}
   	  \includegraphics[height=1.85 cm]{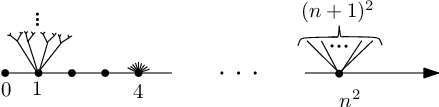} 
   	  \caption{A plane tree with non-zero anchored vertex expansion whose roundabout 
graph has zero anchored vertex expansion.}\label{fig:anchor}
\end{center}
   \end{figure} 
\end{eg}

\section{Degree-weighted energy}\label{sec9}

In this section we prove \autoref{supersuper-trans_mod} already mentioned in the introduction. 
We define the \emph{degree-weighted energy} ${\Ecal_{deg}(f)}$ of a flow $f$ in a graph $G$ to be 

\noindent {$\sum_{v\in V(G)} deg(v) \left(\sum_{e \ni v} |f(e)| \right)^2$}. 

\begin{cor}\label{supersuper-trans}
Let $G$ be a locally finite planar graph that admits a flow $f$ from
some vertex $v$ such that ${\Ecal_{deg}(f)}$
 is finite.
Then $G$ admits a non-constant Dirichlet harmonic function.
\end{cor}

\begin{proof}
By \autoref{thm:UK-trans}, it suffices to show that $\mathring{G}$ is transient. 
Towards this aim, we extend the flow $f$ on $G$ to a flow $g$ on $\mathring{G}$  from some vertex 
$v'$ in the roundabout of $v$ of finite (Dirichlet) energy by assigning values to the edges of the roundabouts. 

For a vertex $z$ of $\mathring{G}$, we denote by $e_z$ the unique edge of $z$ not in any 
roundabout. 
At each roundabout $\mathring{w}$ of a vertex $w\neq v$ of $G$, we have to solve a finite Dirichlet-Problem: we want 
to 
find a function $g_w$  assigning values to the edges of $\mathring{w}$ such that at each vertex 
$z\in \mathring{w}$, the superimposition of $g_w$ with $f$ satisfies \knl\ at all vertices of $\mathring{w}$. As $f$ satisfies \knl\ at $w$, it is easy to see that such a $g_w$ always exists, and it is unique up to adding a multiple of the constant 
flow around $\mathring{w}$. Similarly, we can define  a function $g_v$  on the edges of $\mathring{v}$ such that  the superimposition of $g_v$ with $f$ satisfies \knl\ at all vertices of $\mathring{v}$ except at a single vertex $v'$ of $\mathring{v}$, since $f$ does not satisfy \knl\ at $v$.

We may assume without loss of generality that these $g_w$ satisfy
\labtequ{kw}{$|g_w(k)|\leq \sum_{e\ni w} 
|f(e)|$ for every edge $k$ of $\mathring{w}$,} 
since otherwise we can add a constant flow of intensity $\sum_{e\ni w} 
|f(e)|$ around $\mathring{w}$ to decrease all values of $g_w$; indeed, this is possible because $|g_w(k) - g_w(k')|\leq \sum_{e\ni w} |f(e)|$ holds for every two edges $k,k'$ of $\mathring{w}$  by the definition of $g_w$.

Superimposing $f$ with all the $g_x$'s defines a flow $g$ from $v'$ on $\mathring{G}$. By \eqref{kw}, the energy of $g$
 is bounded, up to a constant depending on $g_v$, by 
${\cal E}(f) + \sum_{w\in V(G)} deg(w) \left(\sum_{e \ni w} |f(e)|\right)^2$, hence it is finite by 
our assumption (where we also used the fact that ${\Ecal_{deg}(f)}< \infty $ implies ${\Ecal (f)}< 
\infty $ by the definitions). Thus $\mathring{G}$ is transient by \autoref{lyons}.
\end{proof}

Given a locally finite graph $G$, for an edge $e=vw$ we let  $r(e)= deg(v)^2+deg(w)^2$. 
We say that $G$ is \emph{super transient} if there is a flow from some vertex with finite $r$-weighted energy, that is, $\sum_{e\in E(G)} f(e)^2 r(e)< \infty$. 
Note that super transience implies transience. Moreover, $G$ is super transient if and only if the 
graph $G[r]$, obtained
 from $G$ by replacing each edge $e$ with a path of length $r(e)$, is transient. The following 
implies \autoref{supersuper-trans_mod}.

\begin{cor}\label{super-trans}
 Every super transient planar locally finite graph $G$ has a non-constant Dirichlet harmonic 
function.
\end{cor}
 
\begin{proof}
By the Cauchy-Schwarz inequality, $\left(\sum_{e \ni v} |f(e)| \right)^2\leq deg(v) \sum_{e \ni v} f(e)^2$. 
Thus this follows from \autoref{supersuper-trans}.
\end{proof}

We remark that if we omit the assumption of planarity, then Corollaries~\ref{supersuper-trans} and 
\autoref{super-trans} become false as the example of the $3$-dimensional grid $\Zbb^3$ shows. 
The next example shows that \autoref{supersuper-trans} is tight in one more sense. 

\begin{eg}\label{grid_minor}
We construct a locally finite planar graph $G\in {\Ocal}_{HD}$ admitting a flow $f$ from some vertex such that for every $\epsilon>0$, we have\\
 ${\cal E}_\epsilon(f)=\sum_{v\in V(G)} {deg(v)}^{(1-\epsilon)} \left(\sum_{e \ni v} |f(e)| 
\right)^2< \infty$. 

In this construction, we rely on the fact that the 2-dimensional grid $\Zbb^2$ contains a subdivision $T$ 
of the infinite binary tree $T_2$
such that edges at level $n$ are subdivided at most $2^n$-times. It is straightforward to construct 
this subdivision $T$ recursively and we leave the details to the reader.
We obtain $G$ from $\Zbb^2$ by contracting for each edge $e$ of $T$ all but one of its subdivision 
edges.

By construction, both $G$ and its dual $G^*$ are 1-ended. Moreover, $G^*$ is obtained from $\Zbb^2$ by deleting edges (again, we are using the fact that deleting an edge in a plane graph corresponds to contracting the same edge in the dual, and vice-versa \cite{Oxley2}).
Thus by \autoref{flow_flow_char}, $G\in \Ocal_{HD}$.

Next, we construct $f$. 
Let $S$ be the subtree of $G$ consisting of those edges of $T$ that are not contracted. By 
construction, the tree $S$ is isomorphic to 
$T_2$.
Let $f$ be the flow from the root of the binary tree $T_2$ which assigns edges at level $n$ the 
value $2^{-n}$. 
Thus $f$ is a flow on $G$ with support $S$.

Let us estimate ${\cal E}_\epsilon(f)$. A vertex $v$ at level $n$ of 
$S$ has degree at most $20\cdot 
2^n$.
Thus\footnote{The constants here do not matter to us, hence we are generous. By choosing the 
edges that remain on the contracted subdivision paths so that they are initial, we can improve the 
above constant `$20$' to the constant `$3$', as then the branch set of the vertex $v$ consists of 
$2^n$ vertices each sending at most 3 edges out of the branch set (most of them actually send at 
most two vertices out). Then the constant `$1000$' below could be improved to `$4\cdot 3=12$'.}  
\[
 {\cal E}_\epsilon(f)\leq 1000\cdot \sum_{n\in \Nbb} 2^n \cdot {2^n}^{(1-\epsilon)} \cdot 2^{-2n}= 
1000\cdot \sum_{n\in \Nbb} {2^{-\epsilon n}}.
\]

Hence ${\cal E}_\epsilon(f)$ is finite, completing this example. 

\end{eg}

\section*{Acknowledgement}
We would like to thank Louigi Addario-Berry for suggesting the use of what we called the plane line 
graph.

\bibliographystyle{plain}
\bibliography{collective1}

\end{document}